\documentclass[11pt]{article}
\usepackage[preprint]{imsart}
\usepackage{amsmath,amssymb,amsthm,amsfonts,amstext,amsbsy,amscd}
\RequirePackage[colorlinks=true,citecolor=blue,urlcolor=blue]{hyperref}

\usepackage{xcolor}
\usepackage{comment}
\usepackage{graphicx}
\usepackage{mathrsfs}

\allowdisplaybreaks[4]
\numberwithin{equation}{section}

\providecolor{darkblue}{rgb}{0.1,0.1,0.6}
\providecolor{darkgreen}{rgb}{0.0,0.4,0.0}

\swapnumbers

\newtheorem{satz}{Satz}[section]

\newtheorem{theorem}[satz]{Theorem}
\newtheorem{proposition}[satz]{Proposition}
\newtheorem{corollary}[satz]{Corollary}
\newtheorem{lemma}[satz]{Lemma}

\newtheorem{remarks}[satz]{Remarks}

\DeclareMathOperator{\E}{{\mathbb E}}

\DeclareMathOperator{\R}{{\mathbb R}}

\DeclareMathOperator{\PP}{{\mathbb P}}
\DeclareMathOperator{\QQ}{{\mathbb Q}}

\DeclareMathOperator{\argmin}{argmin}
\DeclareMathOperator{\Var}{Var} \DeclareMathOperator{\Cov}{Cov}

\providecommand{\eps}{\varepsilon}
\renewcommand{\phi}{\varphi}
\renewcommand{\theta}{\vartheta}
\renewcommand{\subset}{\subseteq}

\renewcommand{\cdot}{{\scriptstyle \bullet} }
\providecommand{\abs}[1]{\lvert #1 \rvert}
\providecommand{\norm}[1]{\lVert #1 \rVert}

\providecommand{\babs}[1]{{\Bigl\lvert #1 \Bigr\rvert}}
\providecommand{\scapro}[2]{\langle #1,#2 \rangle}
\providecommand{\floor}[1]{\lfloor #1 \rfloor}

\newcommand{\tmm}{t}
\newcommand{\tmmpar}{\tmm_{\beta,p,R}(\delta)}

\newcommand{\mtc}{\mathcal}
\newcommand{\mbf}{\mathbf}

\newcommand{\mbe}{\mathbb{E}}

\def\cO{{\mtc{O}}}

\def\cR{{\mtc{R}}}

\newcommand{\wh}[1]{{\widehat{#1}}}

\newcommand{\ind}[1]{{\mbf{1}(#1)}}

\newcommand{\brac}[1]{\left[#1\right]}

\newcommand{\set}[1]{\left\{#1\right\}}

\newcommand{\e}[1]{\mbe\brac{#1}}

\newcommand{\LeftEqNo}{\let\veqno\@@leqno}

\newcounter{nbdrafts}
\makeatletter
\newcommand{\checknbdrafts}{
\ifnum \thenbdrafts > 0
\@latex@warning@no@line{**********************************************************************}
\@latex@warning@no@line{* The document contains \thenbdrafts \space draft note(s)}
\@latex@warning@no@line{**********************************************************************}
\fi}

\makeatother

\begin{document}

\renewcommand{\MR}{\color{blue}}
\newcommand{\GB}{\color{green}}
\newcommand{\MH}{\color{orange}}

\begin{frontmatter}

\title{Early stopping for statistical inverse problems\\ via truncated SVD estimation
}

\runtitle{Early stopping for truncated SVD}

\begin{aug}
  \author{\fnms{Gilles}  \snm{Blanchard}}

  \affiliation{Universit\"at Potsdam, Germany}

  \address{Institut f\"ur Mathematik\\ Universit\"at Potsdam, Germany\\ 
   gilles.blanchard@math.uni-potsdam.de}

\end{aug}
\begin{aug}
  \author{\fnms{Marc}  \snm{Hoffmann}}

  \affiliation{Universit\'e Paris-Dauphine \& PSL, France}

  \address{CEREMADE\\Universit\'e Paris-Dauphine \& PSL, France 
  \\ hoffmann@ceremade.dauphine.fr}

\end{aug}
\begin{aug}
  \author{\fnms{Markus}  \snm{Rei{\ss}}}

 \runauthor{G. Blanchard, M. Hoffmann and M. Rei{\ss}}

  \affiliation{Humboldt-Universit\"at zu Berlin, Germany}

  \address{Institut f\"ur Mathematik\\ Humboldt-Universit\"at zu Berlin, Germany\\
  mreiss@mathematik.hu-berlin.de}

\end{aug}

\begin{abstract}
We consider truncated SVD (or spectral cut-off, projection) estimators for a prototypical statistical inverse problem in dimension $D$.
Since calculating the singular value decomposition (SVD) only  for the largest singular values is much less costly than the full SVD,
our aim is to select a data-driven truncation level $\widehat m\in\{1,\ldots,D\}$ only based on the knowledge of the first $\widehat m$ singular values and vectors.

 We analyse in detail whether sequential {\it early stopping} rules of this type can preserve statistical optimality.  Information-constrained lower bounds and matching upper bounds for a residual based stopping rule are provided, which give a clear picture in which situation optimal sequential adaptation is feasible. Finally, a hybrid two-step approach is proposed which allows for classical oracle inequalities while considerably reducing numerical complexity.
\end{abstract}



\end{frontmatter}

\section{Introduction and overview of results}

\subsection{Model}

A classical model for statistical inverse problems is the  observation of
\begin{equation} \label{eq:WNM}
Y=A\mu+\delta \dot W
\end{equation}
where $A:H_1\to H_2$ is a linear, bounded operator between real Hilbert spaces $H_1,H_2$, $\mu\in H_1$ is the signal of interest, $\delta>0$ is the noise level and $\dot W$ is a Gaussian white noise in $H_2$, see {\it e.g.} Bissantz {\it et al.} \cite{BHMR}, Cavalier \cite{C} and the references therein. { In any concrete situation the problem is discretised, for instance by using a Galerkin scheme projecting on finite element or other approximation spaces.}
Therefore we can assume $H_1=\R^D$, $H_2=\R^P$ with possibly very large $D$ and $P$. Since the discretisation of $\mu$ is at our choice,  we assume
$D\le P$, and that $A:\R^D\to\R^P$ is one-to-one. We  transform \eqref{eq:WNM} by the singular value decomposition (SVD) of $A$ into the  Gaussian vector observation model
\begin{equation} \label{eq: def seq model}
Y_i = \lambda_i \mu_i+ \delta \eps_i,\;\;i=1,\ldots, D,
\end{equation}
where $\lambda_1\ge \lambda_2\ge\cdots\ge\lambda_D>0$ are the  nonzero singular values of $A$, $(\mu_i)_{1\le i\le D}$ the coefficients of $\mu$ in the orthonormal basis of singular vectors and
$(\eps_i)_{1 \leq i \leq D}$ are independent standard Gaussian random variables. The results will easily extend to subgaussian errors as discussed below.

Working in the SVD representation \eqref{eq: def seq model}, the objective is to recover the signal $\mu = (\mu_i)_{1 \le i \le D}$ with best possible accuracy from the data $(Y_i)_{1 \le i \le D}$. A classical method is to use the truncated SVD  estimators (also called projection or spectral cut-off estimators) $\widehat\mu^{(m)}$, $0\leq m \leq D$, given by
\begin{equation}
  \label{eq:defest}
  \wh{\mu}^{(m)}_i={\bf 1}(i\le m)\lambda_i^{-1}Y_i, \qquad i=1,\ldots,D,
  \end{equation}
which are ordered with decreasing bias and increasing variance (w.r.t. $m$). Choosing a suitable truncation index $\widehat m  = \widehat m(Y)$ from the
observed data is the genuine problem of adaptive model selection.
Typical methods use (generalized) cross validation, see {\it e.g.} Wahba \cite{Wa}, unbiased risk estimation, see {\it e.g.} Cavalier {\it et al.} \cite{CGPT}, penalized empirical risk minimisation, see {\it e.g.} Cavalier and Golubev \cite{CG}, or Lepski's balancing principle for inverse problems, see {\it e.g.} Math\'e and Pereverzev \cite{MP}. They all share the drawback that  the estimators $\widehat \mu^{(m)}$ have first to be computed for all values of $0 \le m \le D$, and then be compared to each other in some way.

In this work, we are motivated by constraints due to the possible obstructive computational complexity of calculating the full SVD in high dimensions. We stress that the initial
discretisation of the observation is generally based on a fixed scheme which
  does not  deliver a representation of the observation vector $Y$ in an SVD basis;
  nor is the full SVD basis of the discretized operator $A$ {\em a priori} available in general, so that it has to be computed on the fly.
  Since the calculation of the largest singular value and its corresponding subspace is much less costly, efficient numerical algorithms rely on {\it deflation} or {\it locking} methods, which achieve the desired accuracy for the larger singular values first and then iteratively achieve the accuracy also for the next smaller singular values. { As a basic example the popular power method can be considered, which usually finds after a few vector-matrix multiplications the top eigenvalue-eigenvector pair (with exponentially small error in the iteration number), so that by iterative application the largest $m$ singular values and vectors are computed with roughly ${\mathcal O}(mD^2)$ multiplications compared to ${\mathcal O}(D^3)$ multiplications for a full SVD in a worst case scenario. We refer to the monograph by Saad \cite{Saad} for a comprehensive exposition of the numerical methods. }

We investigate the possibility of an approach which
is both statistically efficient and  sequential along the SVD
in the following sense: we aim at {\em early stopping} methods, in which
the truncated SVD estimators $\widehat \mu^{(m)}$ for $m=0, 1,\ldots,$ are computed iteratively, a stopping rule decides to stop at some step $\widehat m$ and then  $\widehat \mu^{(\widehat m)}$ is used as the estimator.

More generally, we envision our setting as a  simple and prototypical model
to study the scope of statistical adaptivity using iterative methods, which are widely used in
computational statistics and learning. A notable feature of these methods is that not only the numerical, but also the statistical complexity ({\it e.g.}, measured by the variance) increases with the number of iterations,
 so that early stopping is essential from both points of view.
It is common to use stopping rules based on monitoring the residuals
 because the user has access without substantial additional cost
 to the residual norm. Observe that the computation of the residual norm
   $\norm{Y-A\widehat\mu^{(m)}}^2 = \norm{Y}^2 - \sum_{i=1}^m Y_i^2$ does not
   require the full SVD, but only the knowledge of the $m$ first coefficient and
   of the full norm $\norm{Y}^2$, which is readily available. The properties of such rules have been well studied for deterministic inverse problems (e.g. the discrepancy principle, see Engl {\it et al.} \cite{EHN}). In a statistical setting, minimax optimal solutions along the iteration path have been identified in different settings, see {\it e.g.} Yao {\it et al.} \cite{YRC} for gradient descent learning,  Blanchard and Math\'e \cite{BM} for conjugate gradients, Raskutti, Wainwright and Yu \cite{RW} for (reproducing) kernel learning and B\"uhlmann and Hothorn \cite{BH} for the application to $L^2$-boosting.  All these methods stop at a fixed iteration step\footnote{ The stopping rule in \cite{RW} for a random design setting is data-dependent because the distribution of
     the design is unknown. The stopping iteration becomes fixed when the design
     distribution is known, the target function belonging to the unit ball of the kernel space.}, depending on the prior knowledge of the smoothness of the unknown solution.

     By contrast, our goal  is to analyse an {\it a posteriori} early stopping rule based on monitoring the residual, which corresponds to proposals by practitioners in the absence of prior smoothness information.
      { An analysis of such a stopping rule for quite general spectral estimators like Landweber iteration  is provided in
the companion paper \cite{BHR2018}. Although the general results can also be applied to the truncated SVD method, we exhibit a more transparent analysis for this prototypical method which gives more satisfactory results: we establish coherent lower  bounds and we obtain adaptivity in strong norm via the oracle property, while for more general spectral estimators  only rate results over Sobolev-type classes can be achieved. Moreover, a hybrid two-step procedure enjoys full adaptivity for the truncated SVD-method.}

\subsection{Non-asymptotic oracle approach}

Our approach is a priori non-asymptotic and concentrates on oracle optimality analysis for
individual signals. The  oracle approach compares the error of $\widehat \mu^{(\widehat m)}$ to the minimal error among $(\widehat\mu^{(m)})_m$ for any signal $\mu$ individually, which entails optimal adaptation in minimax settings, see e.g. Cavalier  \cite{C}.

The  risk (mean integrated  squared error)
for a fixed truncated SVD estimator $\widehat \mu^{(m)}$ obeys a standard squared bias-variance decomposition
\[
\E\big[\|\widehat \mu^{(m)}-\mu\|^2\big] = B_m^2(\mu)+V_m\,,
\]
where $\norm{\cdot}$ denotes the Euclidean norm in $\R^D$, and
\begin{align}
  B_m^2(\mu)  &:= \E\big[\|\E[\widehat \mu^{(m)}]-\mu\|^2\big]  =
 \textstyle  \sum_{i = m+1}^D \mu_i^2, \label{eq: strong bias}\\
   V_m  & := \E\big[\| \widehat\mu^{(m)} - \E[\widehat \mu^{(m)}]\|^2\big] = \textstyle \delta^2\sum_{i = 1}^m \lambda_i^{-2}.
  \label{eq: strong var}
\end{align}
In distinction with the weak norm quantities defined below,
we call $B_m(\mu)$ {\it strong bias} of $\mu$ and $V_m$ {\it strong variance}.

If we have access to the residual squared norm
\begin{equation} \label{def residual}
\textstyle  R_m^2 :=\|Y-A\widehat \mu^{(m)}\|^2 =\norm{Y}^2-\norm{A\widehat\mu^{(m)}}^2= \sum_{i=1}^D (Y_i - \lambda_i \widehat \mu_i^{(m)})^2 = \sum_{i=m+1}^D Y_i^2,
\end{equation}
then  $R_m^2-(D-m)\delta^2$ gives some bias information  due to
\[ \textstyle \E[R_m^2-(D-m)\delta^2]=B_{m,\lambda}^2(\mu),\quad \text{ with }\quad B_{m,\lambda}^2(\mu):=\sum_{i=m+1}^D \lambda_i^2\mu_i^2.
\]
We call $B_{m,\lambda}^2(\mu)$ the {\it weak bias} and similarly $V_{m,\lambda}=m\delta^2$ the {\it weak variance}. They correspond to measuring the error in the {\it weak norm} (or prediction norm) $\norm{v}_\lambda^2:=\norm{Av}^2=\sum_{i=1}^D\lambda_i^2v_i^2$, which usually (always if $\lambda_1< 1$) is  smaller than the {\it strong} Euclidean norm $\norm{\cdot}$. The  squared bias-variance decomposition for the weak risk then reads
$\E\big[\|\widehat \mu^{(m)}-\mu\|^2_\lambda\big] = B_{m,\lambda}^2(\mu)+V_{m,\lambda}$.
Our setting is thus a particular instance of the question raised by Lepski \cite{L} whether adaptation in one loss  (here: weak norm) leads to adaptation in another loss (here: strong norm). Our positive answer for truncated SVD or spectral cut-off estimation will also extend the results by Chernousova {\it et al.} \cite{CGK}.

Intrinsic to the sequential analysis is the fact  that at truncation index $m$ we cannot say anything about the way the bias decreases for larger indices: it may drop to zero at $m+1$ or even stay constant until $D-1$. Even if we knew the exact value of the bias until index $m$, we could not minimise the sum of squared bias and variance sequentially. Instead, we should wait until the squared bias is sufficiently small to equal (approximately) the variance. This leads to the notion of the  {\em strongly balanced oracle}
\begin{align}\label{Eqms}
  m_{\mathfrak s} = m_{\mathfrak s}(\mu) & :=\min\{m\in\{0,\ldots,D\}\,|\, V_m\ge B_m^2(\mu)\},
\end{align}
whose risk is always upper bounded by twice the classical oracle risk, see \eqref{EqBalClassOracle} below.

\subsection{Setting for asymptotic considerations}
\label{se:asymptsetting}

Risk estimates over classes of signals and asymptotics for vanishing noise level $\delta\to 0$ often help to reveal main features. This way, we can also provide lower bounds for sequential estimation procedures and compare them directly to classical minimax convergence rates. In our setting, the magnitude of the discretisation
dimension $D$ plays a central role, so that it is sensible to assume in an asymptotic view that $D=D_\delta \rightarrow \infty$ as $\delta \rightarrow 0$.
As classes of signals, we will consider  the  Sobolev-type ellipsoids
\begin{equation} \label{def:Sobolev ellipsoid_D}
\textstyle H^\beta(R,D):=\{\mu\in\R^D\,|\,\sum_{i=1}^D i^{2\beta}\mu_i^2\le R^2\}, \;\;\beta\ge 0, R>0,
\end{equation}
and we shall use the following polynomial spectral decay assumption
\begin{equation}
  \label{eq:polydecay}
C_A^{-1} i^{-p} \le \lambda_i \le C_A i^{-p}, \qquad 1 \leq i \leq D, \tag{{\bf PSD}($p,C_A$)}
\end{equation}
for $p\ge 0$, $C_A \geq 1$. The spectrum is allowed to change with $D$ and $\delta$, but $p, C_A$ are considered as fixed constants.
Under these assumptions, standard computations yield for $\mu \in H^\beta(R,D)$,
$1 \leq m \leq D$:
\[
B^2_m(\mu) \leq R^2 m^{-2\beta} \,; \qquad V_m \leq C_A^{-2} \delta^2 m^{2p+1}\,.
\]
Balance between these squared bias and variance bounds is obtained for
$m$ of the order of the minimax truncation ``time''
\begin{equation}\label{Eqtmm}
\tmmpar:= (R^{-1}\delta)^{-2/(2\beta+2p+1)},
\end{equation}
provided
the condition $D \geq \tmmpar$ holds. This gives rise to the risk rate
\[
{\cal R}^*_{\beta,p,R}(\delta) :=  R (R^{-1}\delta)^{2\beta/(2\beta+2p+1)},
\]
which agrees with the optimal minimax rate in the standard Gaussian sequence model
(i.e. $D=\infty$). On the other hand, for $D \lesssim \tmmpar $ the choice
 $m=D$ is optimal on $H^\beta(R,D)$ and the rate degenerates to $ \mathcal{O}(D^{2p+1}\delta^2)$. This situation is indicative of an insufficient discretisation and  will be excluded  from the asymptotic considerations.

\subsection{Overview of results}

Our results consist of lower and upper bounds for sequentially adaptive stopping rules.
 The stopping rules  permitted are
most conveniently described in terms of stopping times with respect to an appropriate filtration.
Introduce the {\em frequency filtration}
\begin{equation}
  \label{def frequency filtration}
\mathcal F_m := \sigma\big( \widehat\mu^{(0)},\ldots,\widehat\mu^{(m)}\big)= \sigma\big(Y_1,\ldots, Y_m\big),
\end{equation}
 $\mathcal F_0$ being the trivial sigma-field. Stopping rules with respect to the filtration $\mathcal F=(\mathcal F_m)_{0\le m\le D}$ must decide whether to halt and output $\widehat\mu^{(m)}$
based only on the information of the first $m$ estimators.
Statistical adaptation will turn out to be essentially impossible for such stopping rules (Section~\ref{sec: freq filtration}).
If the residual \eqref{def residual}
is available at no substantial computational cost, taking this information into account, we define the {\em residual filtration}
\begin{equation} \label{def residual filtration}
\mathcal G_m :=\mathcal F_m \vee \sigma(R_0^2,\ldots,R_m^2)=\mathcal F_m \vee \sigma\big(\|Y\|^2\big),
\end{equation}
which is the filtration $\mathcal F_m$ enlarged by the residuals up to index $m$.

Pushing some technical details aside, the main message conveyed by our lower bounds is that oracle statistical adaptation
with respect to the residual filtration is {\it impossible} for signals $\mu$ such that the strongly balanced oracle
$m_{\mathfrak s}(\mu)$  is $o( \sqrt{D})$.
(Section~\ref{section residual filtration}). On the other hand, we establish in Section~\ref{sec:upper bounds}
that this statement is sharp, in the sense that the simple residual-based stopping rule
\begin{equation}
\label{def early stopping}
  \tau = \min\big\{m\ge m_0\,|\,R_m^2\le \kappa \big\},
\end{equation}
with a proper choice of $\kappa$ and $m_0$
is  statistically adaptive for signals $\mu$ such that $m_{\mathfrak s}(\mu) \gtrsim \sqrt{D}$. Let us stress that by minimax adaptive we always
mean that the procedure attains the optimal rate even among all methods with access to the entire data, that is without information constraints.

Finally, in Section~\ref{se: two-step} we introduce a hybrid two-step approach consisting
of the above stopping rule with  $m_0\sim \sqrt{D} \log D$,
followed by a traditional (non-sequential) model selection procedure over $m\leq m_0$,
in the case where $\tau=m_0$ (immediate stop hinting at an optimal index smaller than $m_0$).
This procedure enjoys full oracle adaptivity
at a computational cost of calculating on average the first $\cO(\max(\sqrt{D}\log D,m_{\mathfrak s}(\mu)))$ singular values, to be compared to the full SVD with $D$ singular values in non-sequential adaptation. Some numerical simulations illustrate the
theoretical analysis. Technical proofs are gathered in an appendix.

\section{Lower bounds} \label{sec: lower bounds}



\subsection{The frequency filtration} \label{sec: freq filtration}

Let $\tau$ be an $\mathcal F$-stopping time, where $\mathcal F$ is the frequency filtration defined in \eqref{def frequency filtration} and let\footnote{We emphasise in the notation the dependence on $\mu$ in the distribution of $\tau$ and the $Y_i$ by adding the subscript $\mu$ when writing the expectation $\E=\E_\mu$ or probability $P=P_\mu$.}
$${\cal R}(\mu,\tau)^2:=\E_\mu[\norm{\widehat\mu^{(\tau)}-\mu}^2].$$
By Wald's identity, we obtain  the simple formula
\begin{equation}\label{EqWald}
\textstyle {\cal R}(\mu,\tau)^2=\E_\mu\big[\sum_{i=\tau+1}^D\mu_i^2+\sum_{i=1}^\tau \lambda_i^{-2}\delta^2\eps_i^2\big]=\E_\mu\big[B_\tau^2(\mu)+V_\tau\big],
\end{equation}
with $B_m^2(\mu)$ and $V_m$  from
\eqref{eq: strong bias}, \eqref{eq: strong var}.
This implies in particular that an oracle stopping time, {\it i.e.},  an optimal $\mathcal F$-stopping time constructed using the knowledge of $\mu$, coincides with the deterministic oracle $\argmin_m\big(B_m^2(\mu)+V_m\big)$ almost surely. The next proposition encapsulates the main argument
  for the lower bound and merely relies on a two-point analysis.
It clarifies that if the stopping time $\tau$ yields a squared risk comparable to
the optimally balanced risk for a given signal $\mu$, then
this signal can be changed arbitrarily to $\bar\mu$ after the index $\floor{3Cm_{\mathfrak s}(\mu)}$, while the risk for the rule $\tau$ always stays larger than the squared bias of that part -- which can be made
arbitrarily large by ``hiding'' signal in large-index coefficients.

\begin{proposition}
\label{prop:case1lower}
Let $\mu,\bar\mu\in\R^D$ with $\mu_i=\bar\mu_i$ for all $i\le i_0$ and $i_0\in\{1,\ldots,D-1\}$. Then  any $\mathcal F$-stopping rule $\tau$ satisfies
\[ {\cal R}(\bar\mu,\tau)^2\ge B_{i_0}^2(\bar\mu)\Big(1-\frac{{\cal R}(\mu,\tau)^2}{V_{i_0+1}}\Big).\]
Suppose ${\cal R}(\mu,\tau)^2\le C {\cal R}(\mu,m_{\mathfrak s})^2$ for the balanced oracle $m_{\mathfrak s}$ in \eqref{Eqms} and some $C\ge 1$. Then
for any $\bar\mu\in\R^D$ with $\bar\mu_i=\mu_i$ for $i\le 3Cm_{\mathfrak s}$ we obtain
\[ {\cal R}(\bar\mu,\tau)^2\ge \tfrac13 B_{\floor{3Cm_{\mathfrak s}}}^2(\bar\mu).\]
\end{proposition}

\begin{proof}
  We use the fact that $(Y_i)_{1 \le i \le i_0}$ has the same law under $P_\mu$ and $P_{\bar\mu}$ and so has ${\bf 1}(\tau \le i_0)$ by the stopping time property of $\tau$. Moreover, thanks to the monotonicity of $m\mapsto V_m$ and $m \mapsto B_m^2(\bar \mu)$, Markov's inequality and  identity \eqref{EqWald}:
\begin{align*}
{\cal R}(\bar\mu,\tau)^2 &\ge \E_{\bar\mu}[B_\tau^2(\bar\mu){\bf 1}(\tau\le i_0)]\\
&=\E_{\mu}[B_\tau^2(\bar\mu){\bf 1}(\tau\le i_0)]\\
&\ge B_{i_0}^2(\bar\mu)P_\mu(\tau\le i_0)\\
&\ge B_{i_0}^2(\bar\mu)(1-P_\mu(V_\tau\ge V_{i_0+1}))\\
&\ge B_{i_0}^2(\bar\mu)\Big(1-\frac{\E_\mu[V_\tau]}{V_{i_0+1}}\Big)\\
&\ge B_{i_0}^2(\bar\mu)\Big(1-\frac{{\cal R}(\mu,\tau)^2}{V_{i_0+1}}\Big).
\end{align*}
The second assertion follows by inserting $i_0=\floor{3Cm_{\mathfrak s}}$ and ${\cal R}(\mu,\tau)^2\le 2CV_{m_{\mathfrak s}}$ together with $V_{m_{\mathfrak s}}/V_{i_0+1}\le m_{\mathfrak s}/(i_0+1)$ since  the singular values $\lambda_i$ are non-increasing.
\end{proof}

In Appendix \ref{SecPrCor22} we use this proposition to provide a result suitable for asymptotic interpretation (we use the notation from
Section~\ref{se:asymptsetting}):

\begin{corollary} \label{cor:case1lower}
Assume \eqref{eq:polydecay} and let $\tau$ be any $\mathcal F$-stopping rule. If there exists $\mu\in H^\beta(R, D)$ with ${\cal R}(\mu,\tau)\le C_\mu {\cal R}^*_{\beta,p,R}(\delta)$, then for any  $\alpha\in [0,\beta]$, $\bar R\ge 2R$, there exists $\bar\mu\in H^\alpha(\bar R, D)$ such that
\[{\cal R}(\bar\mu,\tau)\ge c_1  \bar R(R^{-1}\delta)^{2\alpha/(2\beta+2p+1)}\,,
\]
provided $D\ge c_2 \tmmpar$. The constants $c_1,c_2>0$ only depend on $C_\mu$ and $C_A$.
\end{corollary}

The conclusion for impossible rate-optimal adaptation is a direct consequence of Corollary \ref{cor:case1lower}: since for any $\alpha<\beta$ the rate $\delta^{2\alpha/(2\beta+2p+1)}$ is suboptimal, no ${\cal F}$-stopping rule can adapt over Sobolev classes with different regularities.
Finally, the rate $\bar R(R^{-1}\delta)^{2\alpha/(2\beta+2p+1)}$ is attained by a deterministic stopping rule that stops at the oracle frequency for $H^\beta(R, D)$, so that the lower bound is in fact a sharp {\it no adaptation} result.

\subsection{Residual filtration} \label{section residual filtration}

We start with a key lemma, similar in spirit to the first
step in the proof of Proposition~\ref{prop:case1lower}, but valid for an
arbitrary random $\tau$. { Here and in the sequel the numerical values are not optimised, but give rise to more transparent proofs and convey some intuition for the worst case order of magnitude.} The proof is delayed until Appendix \ref{proof:lem:lowerstop}.
\begin{lemma}
\label{lem:lowerstop}
Let $\tau = \tau\big((Y_i)_{1 \le i \le D}\big)\in\{0,\ldots,D\}$ be an arbitrary (measurable) data-dependent index.
Then for any $m\in\{1,\ldots,D\}$ the following implication holds true:
\[
V_m \ge 200\, \cR(\mu,\tau)^2 \Rightarrow P_\mu(\tau \ge m)\le 0.9.
\]
\end{lemma}
For $\mathcal G$-stopping rules, where $\mathcal G$ is the residual filtration defined in \eqref{def residual filtration}, we deduce the following lower bound, again based on a two-point argument:

\begin{proposition}
\label{prop:case2lower}
Let $\tau$ be an arbitrary ${\cal G}$-stopping rule. Consider  $\mu\in\R^D$ and $i_0\in\{1,\ldots,D\}$ such that $V_{i_0+1}\ge\, 200{\cal R}(\mu,\tau)^2$.  Then
\[ {\cal R}(\bar\mu,\tau)^2\ge 0.05 B_{i_0}^2(\bar\mu)\]
holds for any $\bar\mu \in\R^D$ that satisfies
 \begin{enumerate}
 \item $\mu_i=\bar\mu_i$ for all $i\le i_0$,
  \item the weak bias bound $\abs{B_{i_0,\lambda}^2(\bar\mu)-B_{i_0,\lambda}^2(\mu)}\le 0.05\frac{\sqrt{D-i_0}}{2}\delta^2$ and
  \item $B_{i_0,\lambda}(\mu) + B_{i_0,\lambda}(\bar\mu)\ge 5.25\delta$.
 \end{enumerate}
 Suppose that ${\cal R}(\mu,\tau)^2\le C_\mu {\cal R}(\mu,m_{\mathfrak s})^2$ holds with some $C_\mu\ge 1$. Then any $i_0\ge 400C_\mu m_{\mathfrak s}$ will satisfy the initial requirement.
\end{proposition}

\begin{proof}
First, we lower bound the risk of $\bar\mu$ by its bias on $\{\tau\le i_0\}$ and then transfer to the law of $\tau$ under $P_\mu$, using the total variation distance on ${\cal G}_{i_0}$:
\begin{align*}
{\cal R}(\bar\mu,\tau)^2 &\ge \E_{\bar\mu}[B_\tau^2(\bar\mu){\bf 1}(\tau\le i_0)]\\
&\ge B_{i_0}^2(\bar\mu)P_{\bar\mu}(\tau\le i_0)\\
&\ge B_{i_0}^2(\bar\mu)\big(P_{\mu}(\tau\le i_0)-\norm{P_\mu-P_{\bar\mu}}_{TV({\cal G}_{i_0})}\big).
\end{align*}
By Lemma \ref{lem:lowerstop} we infer $P_{\mu}(\tau\le i_0)\ge 0.1$.
Denote $W_{i_0} = (Y_1,\ldots,Y_{i_0})$. Since the law of $W_{i_0}$
is identical under $P_\mu$ and $P_{\bar\mu}$, and $W_{i_0}$ is independent
of $R^2_{i_0}$ for both measures, the total variation distance between $P_\mu$ and $P_{\bar\mu}$ on ${\cal G}_{i_0}$ equals the total variation distance between the respective laws of the scaled residual $\delta^{-2}R_{i_0}^2$. For $\vartheta \in \R^D$, let $\PP_K^\theta$ be the non-central $\chi^2$-law of $X_\theta=\sum_{k=1}^K(\theta_k+Z_k)^2$ with $Z_k$ independent and standard Gaussian. With $K=D-i_0$, $\theta_k=\delta^{-1}\lambda_{i_0+k}\mu_{i_0+k}$, $\bar\theta_k=\delta^{-1}\lambda_{i_0+k}\bar\mu_{i_0+k}$, the total variation distance between the respective laws of the scaled residual $\delta^{-2}R_{i_0}^2$ exactly equals $\|\PP_K^\theta-\PP_K^{\bar\theta}\|_{TV}$. By Lemma \ref{LemTVGamma} in the Appendix, taking account of $\norm{\theta}=\delta^{-1}B_{i_0,\lambda}(\mu)$ and similarly for $\norm{\bar\theta}$, we infer from (c) the simplified bound
\[ \norm{P_\mu-P_{\bar\mu}}_{TV({\cal G}_{i_0})}\le \frac{2\abs{B_{i_0,\lambda}^2(\bar\mu)-B_{i_0,\lambda}^2(\mu)}}{\delta^2\sqrt{D-i_0}}.\]
Under our assumption on $\bar\mu$, this is at most 0.05, and the inequality follows. From ${\cal R}(\mu,\tau)^2\le 2C_\mu V_{m_{\mathfrak s}}$ and $V_{i_0+1}/V_{m_{\mathfrak s}}\ge (i_0+1)/m_{\mathfrak s}$, the last statement follows.
\end{proof}

In comparison with the frequency filtration, the main new hypothesis is that at $i_0$ the weak bias of $\bar\mu$ is sufficiently close to that of $\mu$, while the lower bound is still expressed in terms of the strong bias. This is natural since the bias only appears in weak form in the residuals, while the risk involves the strong bias. Condition (c) is just assumed to simplify the bound. To obtain valuable counterexamples, $\bar\mu$ is usually chosen at maximal weak bias distance of $\mu$ allowed by (b), so that (c) is always satisfied in the interesting cases where $\sqrt{D-i_0}$ is not small.

Considering the behaviour over Sobolev-type ellipsoids, we obtain in Appendix \ref{SecCor25} a lower bound result comparable to Corollary \ref{cor:case1lower} for the frequency filtration.

\begin{corollary} \label{cor: residual LB1}
    Assume \eqref{eq:polydecay} and let $\tau$ be any $\mathcal G$-stopping time.
  If there exists $\mu\in H^\beta(R, D)$ with ${\cal R}(\mu,\tau)\le C_\mu
  {\cal R}^*_{\beta,p,R}(\delta)$, then for any  $\alpha \in [0,\beta]$\, and  $\bar R\ge 2R$,
there exists $\bar\mu\in H^\alpha(\bar R, D)$ such that
\[  {\cal R}(\bar\mu,\tau)\ge c_1 \bar R\min\Big(\big(\bar R^{-1} \delta D^{1/4}\big)^{2\alpha/(2\alpha+2p)}
,\,(R^{-1}\delta)^{2\alpha/(2\beta+2p+1)}\Big),\]
provided $R^{-1}\delta\le c_2$ and
$D\ge c_3 \tmm_{\alpha-\frac{1}{4},p,\bar R}(\delta)$\,.
The constants  $c_1,c_3>0$, and $c_2\in(0,1]$ depend only on
$C_\mu,C_A,\alpha,p$.
\end{corollary}

The form of the lower bound is transparent: as in the case of the frequency filtration, the
sub-obtimal rate $\bar R(R^{-1}\delta)^{2\alpha/(2\beta+2p+1)}$ is  the
one attained by a deterministic rule that stops at the oracle frequency for $H^\beta(R, D)$, whereas $\bar R\big(\bar R^{-1} \delta D^{1/4}\big)^{2\alpha/(2\alpha+2p)}$ is  the size of a signal that may be hidden in the noise of the residual,  {\em i.e.},
is not detected with positive probability by any test,
thus also leading to erroneous early stopping. Note that for the direct problem ($p=0$), the latter quantity is just $\delta D^{1/4}$, which is exactly the critical signal strength in nonparametric testing, see Ingster and Suslina \cite{IS}, while for $p>0$, it reflects the interplay between the weak bias part in the residual and the strong bias part in the risk within the Sobolev ellipsoid.


Corollary \ref{cor: residual LB1} implies in turn explicit constraints for the maximal Sobolev regularity to which a $\mathcal G$-stopping rule can possibly adapt. Here, we argue asymptotically and let explicitly $D=D_\delta$ tend to infinity as the noise level $\delta$ tends to zero.
In this setting, a stopping rule $\tau$ is to be understood as a family of
stopping rules that depend on the knowledge of $D$ and $\delta$.

\begin{corollary} \label{CorSobLB}
  Assume \eqref{eq:polydecay}. Let $\beta_+>\beta_-\ge 0$, $R_+,R_->0$.
  Suppose that there exists a ${\cal G}$-stopping rule $\tau$ such that ${\cal R}(\mu,\tau)\le C \mathcal{R}^*_{\beta,p,R}(\delta)$
  holds for some $C>0$, all $\delta>0$ small enough, and for every  $\mu\in H^{\beta}(R, D_\delta)$,
  simultaneously for $(\beta,R)\in\{(\beta_-,R_-),(\beta_+,R_+)\}$. Then the rate-optimal truncation  time for $H^{\beta_-}(R_-, D_\delta)$ must satisfy
  $\sqrt{D_\delta}={\mathcal O}(\tmm_{\beta_-,p,R_-}(\delta))$ as $\delta\to 0$ (all other parameters being fixed).

In particular, if a ${\cal G}$-stopping rule $\tau$ is rate-optimal over $H^\beta(R, D_\delta)$ for $\beta\in[\beta_{min},\beta_{max}]$, $\beta_{max}>\beta_{min}\ge 0$, and some $R>0$, then we necessarily must have $\beta_{max}\le \liminf_{\delta \to 0}\frac{\log \delta^{-2}}{\log D_\delta}-p-1/2$.
\end{corollary}

\begin{proof}
  In this proof we denote by '$\lesssim$', '$\gtrsim$' inequalities holding up to
   factors  depending  on $C_A,p,\beta_+,\beta_-,R_-,R_+$.
We apply Corollary \ref{cor: residual LB1} with $\beta=\beta_+$, $\alpha=\beta_-$ and $\bar R=R_-$, $R=\min(R_+,\bar R/2)$. Because of $\delta^{-2/(2\beta_-+2p+1/2)}\le \delta^{-4/(2\beta_-+2p+1)}=o(D_\delta)$, the conditions are fulfilled for sufficiently small $\delta>0$ and we conclude ($R_+,R_-$ are fixed)
\[ \exists\bar\mu\in H^{\beta_-}(R_-, D):\;  {\cal R}(\bar\mu,\tau)\gtrsim \min\Big(\big( \delta D_\delta^{1/4}\big)^{2\beta_-/(2\beta_-+2p)}
,\,\delta^{2\beta_-/(2\beta_++2p+1)}\Big).\]
By assumption, that rate must be  ${\mathcal O}(\delta^{2\beta_-/(2\beta_-+2p+1)})$. Since the second term
in the above minimum is of larger order than this, this must imply
$( \delta D_\delta^{1/4})^{2\beta_-/(2\beta_-+2p)}\lesssim \delta^{2\beta_-/(2\beta_-+2p+1)}$, and
further $\sqrt{D_\delta} \lesssim \delta^{-2/(2\beta_-+2p+1)} \lesssim \tmm_{\beta_-,p,R_-}(\delta).$
The first statement is proved.

For the second assertion, we proceed by contradiction and assume $\beta_{max} > \beta_{lim}:=\liminf_{\delta \rightarrow 0}
\frac{\log \delta^{-2}}{\log D_\delta} - p - 1/2$. Choose $\beta_+ = \beta_{max}$ and $\beta_- \in (\beta_{lim},\beta_{max})$.
Then $\beta_- > \beta_{lim}$ implies $\tmm_{\beta_-,p,R_-}(\delta_k)=o(\sqrt{D_{\delta_k}})$
for some sequence $\delta_k \rightarrow 0$, contradicting
the first
part of the corollary.
\end{proof}

For statistical inverse problems with singular values satisfying the polynomial decay \eqref{eq:polydecay} we may choose the maximal dimension $D_\delta \thicksim \delta^{-2/(2p+1)}$ without losing in the convergence rate for a Sobolev ellipsoid of any regularity $\beta\ge 0$, see e.g. Cohen {\it el al.} \cite{CHR}. In fact, we then have the variance
\begin{equation}\label{EqDConserv}
V_{D_\delta}=\delta^2\sum_{i=1}^{D_\delta}\lambda_i^{-2}\thicksim\delta^2 D_\delta^{2p+1}\thicksim 1,
\end{equation}
and the estimator with truncation at the order of $D_\delta$ will not be consistent anyway;
the oracle index is always of order $o(D_\delta)$ whatever the signal regularity.
For this choice of $D_\delta$, optimal adaptation is  only possible if the squared minimax rate is within the interval $[\delta,1]$, faster adaptive rates up to $\delta^2$ cannot be attained.

  Usually, $D_\delta$ will be chosen much smaller, assuming some minimal a priori regularity $\beta_{min}$. The choice $D_\delta \thicksim \delta^{-2/(2\beta_{min}+2p+1)}$ ensures that rate optimality is possible for all (sequence space) Sobolev regularities $\beta\ge\beta_{min}$,
when using either oracle (non-adaptive) rules, or adaptive rules that are not
stopping times. In contrast, any ${\cal G}$-stopping rule can at best adapt over the regularity interval $[\beta_{min},\beta_{max}]$ with $\beta_{max}=2\beta_{min}+p+1/2$ (keeping the radius $R$ of the Sobolev ball fixed). These adaptation intervals, however, are fundamentally understood only when inspecting the corresponding rate-optimal truncation indices $\tmmpar$, which must at least be of order $\sqrt{D_\delta}\thicksim \delta^{-1/(2\beta_{min}+2p+1)}$ in order to distinguish a signal in the residual from the pure noise case.

\section{Upper bounds} \label{sec:upper bounds}

Consider the residual-based stopping rule
$\tau = \min\big\{m\ge m_0\,|\,R_m^2\le \kappa\big\}$ from \eqref{def early stopping}.
Since $R_{m}^2$ is decreasing with $R_D^2=0$, the minimum is attained and we  have $R_{\tau}^2\le\kappa$.

In order to have clearer oracle inequalities, we work with continuous oracle-type truncation indices in $[0,D]$. Interpolating such that $V_{t,\lambda}=t\delta^2$ continues to hold for real $t\in[0,D]$, we set
\[ \widehat\mu^{(t)}_i:=\Big({\bf 1}(i\le \floor{t})+\sqrt{t-\floor{t}}{\bf 1}(i=\floor{t}+1)\Big)\lambda_i^{-1}Y_i,\quad i=1,\ldots,D,\]
and define further
\begin{align*}
R_t^2&=\big(1-\sqrt{t-\floor t}\big)^2Y_{\floor{t}+1}^2+\textstyle\sum_{i=\floor{t}+2}^DY_i^2,\\
B_{t}^2(\mu)&=\big(1-\sqrt{t-\floor t}\big)^2\mu_{\floor{t}+1}^2+\textstyle\sum_{i=\floor{t}+2}^D\mu_i^2,\\ V_{t}&=(t-\floor{t})\delta^2\lambda_{\floor{t}+1}^{-2}+\delta^2\textstyle\sum_{i=1}^{\floor{t}}\lambda_i^{-2},\\
 S_t&=(t-\floor{t})\delta^2\lambda_{\floor{t}+1}^{-2}\eps_{\floor{t}+1}^2+\delta^2\textstyle\sum_{i=1}^{\floor{t}}\lambda_i^{-2}\eps_i^2.
\end{align*}
We thus obtain the following decompositions in a bias and a stochastic error term:
\begin{gather}
\norm{\widehat\mu^{(t)}-\mu}^2=B_t^2(\mu)+S_t+2\delta \big(t-\floor{t}-\sqrt{t-\floor{t}}\big)\lambda_{\floor{t}+1}^{-1}\mu_{\floor{t}+1}\eps_{\floor{t}+1},\label{EqBiasStochDecomp}\\
\E\big[\norm{\widehat\mu^{(t)}-\mu}^2\big] =B_t^2(\mu)+V_t,\quad \E\big[\norm{\widehat\mu^{(\tau)}-\mu}^2\big]=\E\big[B_\tau^2(\mu)+S_\tau\big],\label{EqBiasVarDecomp}
\end{gather}
noting that the last term in \eqref{EqBiasStochDecomp} has expectation zero for deterministic $t$ and vanishes for the integer-valued random time $\tau$.
Analogously, the linear interpolations for bias and variance in weak norm are defined.
Thus, the continuously interpolated residual has expectation
\begin{align}
 \E[R_t^2]&=B_{t,\lambda}^2(\mu)+\Big(\big(1-\sqrt{t-\floor t}\big)^2+(D-\floor{t}-1)\Big)\delta^2 \notag \\
 &=B_{t,\lambda}^2(\mu)-V_{t,\lambda}+D\delta^2-2\big(\sqrt{t-\floor t}-(t-\floor t)\big)\delta^2. \label{expresidual}
\end{align}
Integrating the last interpolation error term into the definition, we define the
{\it oracle-proxy index} $t^\ast\in[m_0,D]$ as
\[\textstyle  t^\ast=\inf\Big\{t\ge m_0\,\Big|\,\E_\mu[R_t^2]\le\kappa-2\big(\sqrt{t-\floor t}-(t-\floor t)\big)\delta^2\Big\} .\]
Then by continuity $ \E[R_{t^\ast}^2]=\kappa-2(\sqrt{t_\ast-\floor{t_\ast}}-(t_\ast-\floor{ t_\ast}))\delta^2\in[\kappa-\frac12\delta^2,\kappa]$ holds in the case $t^\ast>m_0$, implying
\begin{equation}\label{EqRtstarkappa}
\kappa=D\delta^2+B_{t^\ast,\lambda}^2(\mu)-V_{t^\ast,\lambda}.
\end{equation}
For $t^\ast=m_0$ we still have $D\delta^2+B_{t^\ast,\lambda}^2(\mu)-V_{t^\ast,\lambda}\le\kappa$.

Let us finally define the {\it weakly and strongly balanced  oracles} $t_{\mathfrak w}$ and $t_{\mathfrak s}$ in a continuous manner:
\begin{align*}
 t_{\mathfrak w} =t_{\mathfrak w}(\mu) &=\inf\{t\ge m_0\,|\, B_{t,\lambda}^2(\mu)\le V_{t,\lambda}\}\in[m_0,D],\\
 t_{\mathfrak s} =t_{\mathfrak s}(\mu) &=\inf\{t\ge m_0\,|\, B_t^2(\mu)\le V_t\}\in[m_0,D].
 \end{align*}
While the balanced oracles are the natural oracle quantities we try to mimic by early stopping, they should be compared to the classical oracles. Since $t\mapsto B_t^2(\mu)$ is decreasing and $t\mapsto V_t$ is increasing, we derive
\begin{equation}\label{EqBalClassOracle}
\inf_{t\in[m_0,D]}\E\big[\norm{\widehat\mu^{(t)}-\mu}^2\big]
\ge \inf_{t\in[m_0,D]} \max(B_t^2(\mu),V_t)
\ge V_{t_{\mathfrak s}}\ge \tfrac12\E\big[\norm{\widehat\mu^{(t_{\mathfrak s})}-\mu}^2\big],
\end{equation}
noting  $ B_{t_{\mathfrak s}}^2(\mu)\le V_{t_{\mathfrak s}}$.

\subsection{Upper bounds in weak norm} \label{upper in weak norm}

{ The following is an analogue of Proposition 2.1 in \cite{BHR2018}, but includes a discretisation error for the discrete time stopping rule $\tau$.}

\begin{proposition}\label{PropMainBound}
The  balanced oracle inequality  in weak norm
\begin{equation} \label{eq:propmainbound}
\E\big[\norm{\widehat\mu^{(\tau)}-\widehat\mu^{(t^\ast)}}_{\lambda}^2\big] \le \sqrt{2D}\delta^2+2\delta B_{t^\ast,\lambda}(\mu)+\Delta_\tau(\mu)^2
\end{equation}
holds with the discretisation error
\[ \Delta_\tau(\mu)=\max_{i\ge \floor{t^\ast}+1} \abs{\lambda_i\mu_i}+4\delta\Big(\big(\log(\sqrt 2D)\big)^{1/2}+1\Big).\]
\end{proposition}

\begin{proof}
The main  argument is completely deterministic. For $\tau>t^\ast\ge m_0$ we obtain by $R_\tau^2+Y_\tau^2=R_{\tau-1}^2>\kappa\ge \E[R_{t^\ast}^2]$:
\begin{align*}
\norm{\widehat\mu^{(t^\ast)}-\widehat\mu^{(\tau)}}_{\lambda}^2&=(1-\sqrt{t^\ast-\floor{t^\ast}})^2Y_{\floor{t^\ast}+1}^2+\textstyle\sum_{i=\floor{t^\ast}+2}^\tau Y_i^2\\
&=R_{t^\ast}^2-R_\tau^2< R_{t^\ast}^2-\E[R_{t^\ast}^2]+Y_\tau^2.
\end{align*}
For $t^\ast>\tau\ge m_0$ we use $t^\ast-\floor{t^\ast}\le 1-(1-\sqrt{t^\ast-\floor{t^\ast}})^2$ and $R_{\tau}^2\le \kappa\le \E[R_{t^\ast}^2]+\frac12\delta^2$:
\begin{align*}
 \norm{\widehat\mu^{(t^\ast)}-\widehat\mu^{(\tau)}}_{\lambda}^2&=(t^\ast-\floor{t^\ast})Y_{\floor{t^\ast}+1}^2+\textstyle\sum_{i=\tau+1}^{\floor{t^\ast}}Y_i^2\\
&\le R_\tau^2-R_{t^\ast}^2\le \E[R_{t^\ast}^2]-R_{t^\ast}^2+\tfrac12\delta^2.
\end{align*}
Consequently, we find
\begin{align*}
&  \E\big[\norm{\widehat\mu^{(t^\ast)}-\widehat\mu^{(\tau)}}_{\lambda}^2\big]\\
&\le \E\Big[\babs{\sum_{i=\floor{t^\ast}+1}^D\gamma_i\big(\delta^2(\eps_i^2-1)+2\lambda_i\mu_i\delta\eps_i\big)}\Big]+\E\Big[\max_{i\ge\floor{t^\ast}+1}Y_i^2\Big]+\tfrac12\delta^2,
\end{align*}
with $\gamma_i=1$ for $i> \floor{t^\ast}+1$ and $\gamma_i=(1-\sqrt{t^\ast-\floor{t^\ast}}\big)^2$ for $i=\floor{t^\ast}+1$.
The maximal inequality in Corollary 1.3 of \cite{Ts} implies
\begin{align*}
\E\Big[\max_{i\ge\floor{t^\ast}+1}Y_i^2\Big]
&\le \Big(\max_{i\ge \floor{t^\ast}+1}\abs{\lambda_i\mu_i}+4\delta\big(\log\big(\sqrt 2(D-\floor{t^\ast})\big)\big)^{1/2}\Big)^2,
\end{align*}
which is smaller than $\Delta_\tau(\mu)^2-\tfrac12\delta^2$.
By bounding 
the main term via 
 Jensen's 
inequality, 
using $\Var(\eps_i^2)=2$, $\Cov(\eps_i^2,\eps_i)=0$,
this gives
\begin{align*}
\E_\mu\big[\norm{\widehat\mu^{(t^\ast)}-\widehat\mu^{(\tau)}}_{\lambda}^2\big]
&\le \Big(2(D-t^\ast)\delta^4+4\delta^2 B_{t^\ast,\lambda}^2(\mu)\Big)^{1/2}+\Delta_\tau(\mu)^2,
\end{align*}
and thus by $\sqrt{A+B}\le \sqrt A+\sqrt B$, $A,B\ge 0$, the asserted inequality.
\end{proof}

Remark that the proof only relies on the moments of $(\eps_i)$ up to fourth order and a maximal deviation inequality,   so that an extension to sub-Gaussian distributions is straightforward. More heavy-tailed distributions can be treated at the cost of a looser bound on $\Delta_\tau(\mu)$.

So far, the choice of $\kappa$ has not been addressed. The identity \eqref{EqRtstarkappa} shows that the choice $\kappa=D\delta^2$ balances weak squared bias and variance exactly such that $t^\ast=t_{\mathfrak w}$. In practice, however, we might have to estimate the noise level $\delta^2$, or we prefer a larger threshold $\kappa$ to reduce numerical complexity. Therefore,  precise bounds for general $\kappa$ between the oracle-proxy and the weakly balanced errors in weak norm are useful.

\begin{lemma}\label{LemTransferWtstarts}
We have
\[ (B_{t^\ast,\lambda}^2(\mu)-B_{t_{\mathfrak w},\lambda}^2(\mu))_+\le (\kappa-D\delta^2)_+,\quad (V_{t^\ast,\lambda}-V_{t_{\mathfrak w},\lambda})_+\le (D\delta^2-\kappa)_+,
\]
so that
\[ \E[\norm{\widehat\mu^{(t^\ast)}-\mu}_{\lambda}^2] \le \E[\norm{\widehat\mu^{(t_{\mathfrak w})}-\mu}_{\lambda}^2]+\abs{\kappa-D\delta^2}.\]
\end{lemma}

\begin{proof}
Suppose $t_{\mathfrak w}> t^\ast\ge m_0$. Then $V_{t^\ast,\lambda}<V_{t_{\mathfrak w},\lambda}$ and from $\kappa\ge B_{t^\ast,\lambda}^2(\mu)+D\delta^2-V_{t^\ast,\lambda}$ (see \eqref{EqRtstarkappa} and afterwards), $V_{t_{\mathfrak w},\lambda}=B_{t_{\mathfrak w},\lambda}^2(\mu)$ we deduce
\[ B_{t^\ast,\lambda}^2(\mu)\le V_{t^\ast,\lambda}+\kappa-D\delta^2< V_{t_{\mathfrak w},\lambda}+\kappa-D\delta^2 =  B_{t_{\mathfrak w},\lambda}^2(\mu)+\kappa-D\delta^2.
\]
Conversely, for $t^\ast>t_{\mathfrak w}\ge m_0$ we have $B_{t^\ast,\lambda}^2(\mu)\le B_{t_{\mathfrak w},\lambda}^2(\mu)$ as well as \eqref{EqRtstarkappa} and $V_{t_{\mathfrak w},\lambda}\ge B_{t_{\mathfrak w},\lambda}^2(\mu)$, so that
\[ V_{t^\ast,\lambda}= B_{t^\ast,\lambda}^2(\mu)-\kappa+D\delta^2\le B_{t_{\mathfrak w},\lambda}^2(\mu)-\kappa+D\delta^2 \le  V_{t_{\mathfrak w},\lambda}-\kappa+D\delta^2.
\]
This gives the result.
\end{proof}

Remark that the weak variance control of Lemma \ref{LemTransferWtstarts} implies directly $(t^\ast-t_{\mathfrak w})_+\le (D-\kappa\delta^{-2})_+$. From the inequalities $B_t^2(\mu)\ge \lambda_{\floor{t}}^{-2}B_{t,\lambda}^2(\mu)$ and $V_t\le \lambda_{\floor{t}}^{-2}V_{t,\lambda}$ we infer further $t_{\mathfrak w}\le t_{\mathfrak s}$,  and thus it always holds
\begin{equation}\label{Eqtoracles}
t^\ast-(D-\kappa\delta^{-2})_+\le t_{\mathfrak w}\le t_{\mathfrak s}.
\end{equation}
As a consequence of the preceding two results, we obtain directly a weakly balanced oracle inequality with error terms of order $\sqrt D\delta^2$, provided $\abs{\kappa-D\delta}^2$ is at most of that order:

\begin{theorem}
We have
\begin{align*}
\E\big[\norm{\widehat\mu^{(\tau)}-\mu}_{\lambda}^2\big] &\le C \Big( \E\big[\norm{\widehat\mu^{(t_{\mathfrak w})}-\mu}_{\lambda}^2\big]  +\sqrt D \delta^2+\abs{\kappa-D\delta^2}\Big)\\
&\le C \Big( 2\min_{t\in[m_0,D]}\E\big[\norm{\widehat\mu^{(t)}-\mu}_{\lambda}^2\big] + \sqrt D \delta^2+\abs{\kappa-D\delta^2}\Big)
\end{align*}
 with a numerical constant $C>0$.
\end{theorem}

\begin{proof}
For the first bound use
\[ \E\big[\norm{\widehat\mu^{(\tau)}-\mu}_{\lambda}^2\big]\le 2\E\big[\norm{\widehat\mu^{(t_{\mathfrak w})}-\mu}_{\lambda}^2\big]+2\E\big[\norm{\widehat\mu^{(\tau)}-\widehat\mu^{(t_{\mathfrak w})}}_{\lambda}^2\big]
\]
 and apply Proposition~\ref{PropMainBound} and Lemma~\ref{LemTransferWtstarts} with the estimates $2\delta B_{t^\ast,\lambda}(\mu)\le \delta^2+B_{t^\ast,\lambda}^2(\mu)$, $\Delta_\tau(\mu)^2\lesssim B_{t^\ast,\lambda}^2(\mu)+\sqrt D \delta^2$, $B_{t^\ast,\lambda}^2(\mu)\le B_{t_{\mathfrak w},\lambda}^2(\mu)+\abs{\kappa-D\delta^2}$ ('$\lesssim$' denotes an inequality up to a numerical factor). The second bound  follows exactly as \eqref{EqBalClassOracle}.
\end{proof}

In weak norm, we have thus obtained a completely general oracle inequality for our early stopping rule. In view of the lower bounds, the "residual term" of order $\sqrt D\delta^2$, which is much larger than the usual parametric order $\delta^2$, is unavoidable. This will be developed further in the strong norm error analysis.

\subsection{Upper bounds in strong norm} \label{upper in strong norm}


%
%

In Appendix \ref{SecProp35} we derive exponential bounds for $P(R_m^2\le\kappa)$, $m<t^\ast$, in terms of the weak bias and deduce by partial summation the following weak bias deviation inequality:

\begin{proposition}\label{PropBiasbound}
We have
\[\E\big[(B_{\tau,\lambda}^2(\mu)-B_{t^\ast,\lambda}^2(\mu))_+\big]\le \big(17\sqrt{D}+64\big)\delta^2+B_{t^\ast,\lambda}^2(\mu) D^{-1/2}.
\]
\end{proposition}

This is the probabilistic basis for the main bias oracle inequality.


\begin{proposition}\label{PropBiastautstar}
We have the balanced oracle inequality for the strong bias
\[
  \E[(B_\tau^2(\mu)-B_{t_{\mathfrak s}}^2(\mu))_+] \le  81\lambda_{\floor{t_{\mathfrak s}}+1}^{-2}\delta^2\Big( t_{\mathfrak s}+\sqrt D +(\kappa\delta^{-2}-D)_+\Big) .
\]
\end{proposition}

\begin{proof}
On  the event $\{\tau  \ge t_{\mathfrak s}\}$ we have $B_{\tau}^2(\mu)\le B_{t_{\mathfrak s}}^2(\mu)$. On $\{\tau< t_{\mathfrak s}\}$ we have
\[B_{\tau}^2(\mu)- B_{t_{\mathfrak s}}^2(\mu)\le \lambda_{\floor{t_{\mathfrak s}}+1}^{-2}(B_{\tau,\lambda}^2(\mu)- B_{t_{\mathfrak s},\lambda}^2(\mu)),
\]
using the fact that only coefficients up to index $\floor{t_{\mathfrak s}}+1$ enter into the bias differences.
From the weak bias control given by Proposition \ref{PropBiasbound}
it follows that
\begin{align*}
\E[(B_\tau^2(\mu)&-B_{t_{\mathfrak s}}^2(\mu))_+]\\
 & \le  \lambda_{\floor{t_{\mathfrak s}}+1}^{-2}\Big(\E[(B_{\tau,\lambda}^2(\mu)-B_{t^\ast,\lambda}^2(\mu))_+]+(B_{t^\ast,\lambda}^2(\mu)-B_{t_{\mathfrak s},\lambda}^2(\mu))_+\Big)\\
& \le  \lambda_{\floor{t_{\mathfrak s}}+1}^{-2}\Big( (17\sqrt D+64)\delta^2+(1+D^{-1/2})B_{t^\ast,\lambda}^2(\mu)\Big)\\
& \le  \lambda_{\floor{t_{\mathfrak s}}+1}^{-2}\delta^2\Big( 81\sqrt D+2(t^\ast+\kappa\delta^{-2}-D)\Big) ,
\end{align*}
where in the last line we used  $\kappa\ge B_{t^\ast,\lambda}^2(\mu)+D\delta^2-V_{t^\ast,\lambda}$ (see \eqref{EqRtstarkappa} and afterwards) and $V_{t^\ast,\lambda}=t^\ast\delta^2$. By \eqref{Eqtoracles} we see $t^\ast\le t_{\mathfrak s}+(D-\kappa\delta^{-2})_+$ and the result follows.
\end{proof}

To assess the size of the bias bound, let us assume  the polynomial decay \eqref{eq:polydecay}. Then a Riemann sum approximation yields for any $t\in[1,D]$
\[ \delta^{-2}V_t=\sum_{i=1}^{\floor{t}}\lambda_i^{-2}+(t-\floor{t})\lambda_{\floor{t}+1}^{-2}\ge C_A^{-2}\int_0^t x^{2p}dx= C_A^{-2}(2p+1)^{-1}t^{2p+1}.\]
Noting $t\lambda_{\floor{t}+1}^{-2}\le C_A^{2}(\floor{t}+1)^{2p+1}\le C_A^{2}(2t)^{2p+1}$, we thus obtain
\begin{equation}\label{EqVstar}
\lambda_{\floor{t}+1}^{-2}t\delta^2\le (1+2p)2^{2p+1}C_A^4V_{t}.
\end{equation}
Consequently, we can estimate $\E[(B_{\tau}^2(\mu)-B_{t_{\mathfrak s}}^2(\mu))_+]\lesssim V_{t_{\mathfrak s}}$ in the case $ t_{\mathfrak s}\gtrsim \max(\sqrt D,\kappa\delta^{-2}-D)$. This means that the bias bound is upper bounded by the balanced strong oracle risk.

Let us see by a counterexample that 
$(B_{t^\ast}^2(\mu)-B_{t_{\mathfrak s}}^2(\mu))_+$ can be of the same  order as the strongly balanced risk
itself, meaning that the bound of Proposition \ref{PropBiastautstar} is not too pessimistic in general. Suppose $\kappa=D\delta^2$ ( so that $t^\ast=t_{\mathfrak w}$), $\mu_D\not=0$ and $\delta,D$ such that $t_{\mathfrak s}=D-3/4$. This gives $\mu_D^2/4=B_{t_{\mathfrak s}}^2(\mu)=V_{t_{\mathfrak s}}$. In weak norm we have  $B_{D-1,\lambda}^2(\mu)=\lambda_D^2\mu_D^2$, $V_{D-1,\lambda}=\delta^2(D-1)$ and consequently $t_{\mathfrak w}\le D-1$ if $\lambda_D^2\mu_D^2\le \delta^2(D-1)$. In that case, $B_{t^\ast}^2(\mu)\ge \mu_D^2=4 B_{t_{\mathfrak s}}^2(\mu)$ holds and we must indeed pay a positive factor for using the weak oracle in strong norm.  We can meet the bound $\lambda_D^2\mu_D^2\le \delta^2(D-1)$ under the constraint $\mu_D^2/4=V_{D-3/4}=\delta^2(\lambda_D^{-2}/4+\sum_{i=1}^{D-1}\lambda_i^{-2})$ for instance for $\lambda_i=i^{-p}$ with $p>3/2$ and $D$ sufficiently large.


For the stochastic error  we use in Appendix \ref{SecProp37} exponential inequalities for $P(R_{m-1}^2>\kappa)$, $m>t^\ast$, to obtain the following bound:

\begin{proposition}\label{PropVartautstar} We have the oracle-proxy inequality for the strong norm stochastic error
\[\E[(S_\tau-S_{t^\ast})_+]\le r_{V,\tau}\delta^2\text{ with }r_{V,\tau} :=\min\Big(2\sqrt 3  \sum_{m=\floor{t^\ast}+1}^D \lambda_{m}^{-2}e^{-\frac{(m-1-t^\ast)_+^2}{16D+32\kappa\delta^{-2}}},\,D\Big).
\]
If the polynomial decay condition \eqref{eq:polydecay} is satisfied, then
\begin{align}
 r_{V,\tau}\delta^2 
 &\le  C_pC_A^4\Big(\big((D+\kappa\delta^{-2})/ (t^\ast)^2\big)^{1/2}+ \big((D+\kappa\delta^{-2})/ (t^\ast)^2\big)^{p+1/2}\Big)V_{t^\ast}\label{EqrVtau}
\end{align}
holds with a constant $C_p$, only depending on $p$.
\end{proposition}



\begin{corollary}\label{CorTransfertstarts}
We have the balanced oracle inequality for the stochastic error
\[  
 \E[(S_\tau-S_{t_{\mathfrak s}})_+]\le \Big(r_{V,\tau}+ \lambda_{\floor{t^\ast}+1}^{-2}(D -\kappa \delta^{-2})_+\Big)\delta^2.
\]
\end{corollary}

\begin{proof}
By the monotonicity of $S_t$ and $V_t$ in $t$ we bound
   \[
   \E[(S_\tau-S_{t_{\mathfrak s}})_+] \le \E[(S_\tau-S_{t^*})_+] + \E[(S_{t^*}-S_{t_{\mathfrak s}})_+]
     = \E[(S_\tau-S_{t^*})_+] + (V_{t^*} - V_{t_{\mathfrak s}})_+\,.
  \]
In view of Proposition \ref{PropVartautstar} it suffices to prove $(V_{t^\ast}-V_{t_{\mathfrak s}})_+\le \lambda_{\floor{t^\ast}+1}^{-2}(D\delta^2 -\kappa )_+$.
By definition of the variances, $(V_{t^\ast}-V_{t_{\mathfrak w}})_+\le \lambda_{\floor{t^\ast}+1}^{-2}(V_{t^\ast,\lambda}-V_{t_{\mathfrak w},\lambda})_+$ holds. We apply Lemma \ref{LemTransferWtstarts} and note $V_{t_{\mathfrak w}}\le V_{t_{\mathfrak s}}$ by \eqref{Eqtoracles} to conclude.
\end{proof}

Everything is prepared to prove our main strong norm result.

\begin{theorem} \label{thm_spectral_cutoff}
Assume  $\abs{\kappa-D\delta^2}\le C_\kappa \sqrt D\delta^2$. Then the following balanced oracle inequality holds in strong norm
\begin{align*}
\E\big[\norm{\widehat\mu^{(\tau)}-\mu}^2\big]&\le \E\big[\norm{\widehat\mu^{(t_{\mathfrak s})}-\mu}^2\big]\\
&\quad
+ \Big(81\lambda_{\floor{t_{\mathfrak s}+C_\kappa \sqrt D}+1}^{-2}\big( t_{\mathfrak s} +(1+C_\kappa)\sqrt D\big)+r_{V,\tau}\Big)\delta^2.
\end{align*}
If in addition the polynomial decay condition \eqref{eq:polydecay} is satisfied, then there is a constant $C>0$, only depending on $p,C_A,C_\kappa$, so that
\begin{align}
\E\big[\norm{\widehat\mu^{(\tau)}-\mu}^2\big]&\le C \E\big[\norm{\widehat\mu^{(t_{\mathfrak s}\vee \sqrt D)}-\mu}^2\big].\label{EqMainBound}
\end{align}
\end{theorem}

{
\begin{remarks}\mbox{}
\begin{enumerate}
\item
The impact of the polynomial decay condition \eqref{eq:polydecay} is quite transparent here. If the eigenvalues decay exponentially, $\lambda_i=e^{-\alpha i}$ say, then a factor $e^{2\alpha C_\kappa\sqrt D}$ appears in the balanced oracle inequality, which we then also lose compared to the optimal minimax rate in strong norm. Polynomial decay ensures $\lambda_{\floor{t_{\mathfrak s}+C_\kappa \sqrt D}+1}^{-2}\lesssim \lambda_{t_{\mathfrak s}}^{-2}$ for $t_{\mathfrak s}\ge\sqrt D$. Intuitively, stopping about $\sqrt D$ steps later does not affect the rate under \eqref{eq:polydecay}, but does affect it under exponential singular value decay.
\item Compared with the main Theorem 3.5 in \cite{BHR2018} this is a proper oracle inequality since for the truncated SVD method $t_{\mathfrak w}\le t_{\mathfrak s}$ always holds. Note also that the more direct proof here gives much simpler and tighter bounds.
\end{enumerate}

\end{remarks}
}

\begin{proof}
By \eqref{EqBiasVarDecomp} we have
\[ \E\Big[\norm{\widehat\mu^{(\tau)}-\mu}^2-\norm{\widehat\mu^{(t_{\mathfrak s})}-\mu}^2\Big]\le \E\Big[\big(B_\tau^2(\mu)-B_{t_{\mathfrak s}}^2(\mu)\big)_++ \big( S_\tau-S_{t_{\mathfrak s}}\big)_+\Big].
\]
Combining Proposition \ref{PropBiastautstar} and Corollary \ref{CorTransfertstarts} we thus  obtain
\begin{align*}
\E\big[\norm{\widehat\mu^{(\tau)}-\mu}^2\big]&\le \E\big[\norm{\widehat\mu^{(t_{\mathfrak s})}-\mu}^2\big]\\
&\quad+ 81\lambda_{\floor{t_{\mathfrak s}\vee t^\star}+1}^{-2}\delta^2\big( t_{\mathfrak s}+\sqrt D +\abs{\kappa\delta^{-2}-D}\big)+r_{V,\tau}\delta^2.
\end{align*}
By \eqref{Eqtoracles} we have $t^\ast\le t_{\mathfrak s}+(D-\kappa\delta^{-2})_+$ and the first inequality follows.

Under \eqref{eq:polydecay} we use \eqref{EqrVtau}, $ V_{t^\ast} \lesssim (t_\ast)^{2p+1} \delta^2$, $\kappa \delta^{-2} \lesssim D$ and $t^\ast\le t_{\mathfrak s}+C_\kappa\sqrt D$ to further bound
\[ r_{V,\tau}\lesssim t_{\mathfrak s}^{2p}\sqrt{D}+D^{p+1/2},
\]
with a factor depending on $p,C_A,C_\kappa$. Finally, note
\begin{align*}
\E[\norm{\widehat\mu^{(t_{\mathfrak s})}-\mu}^2] &\le 2V_{t_{\mathfrak s}}\le  2V_{t_{\mathfrak s}\vee\sqrt D}\le 2\E[\norm{\widehat\mu^{(t_{\mathfrak s}\vee\sqrt D)}-\mu}^2],\\
 V_{t_{\mathfrak s}\vee \sqrt D} &\thicksim (t_{\mathfrak s}\vee\sqrt D)^{2p+1}\delta^2,
\end{align*}
and apply $\lambda_{\floor{t_{\mathfrak s}+C_\kappa \sqrt D}+1}^{-2}\lesssim (t_{\mathfrak s}\vee\sqrt D)^{2p}$ to deduce the second bound.
\end{proof}

Again, the proof only relies on concentration bounds for the residuals $R_m^2$ and easily extends to sub-Gaussian errors.
Let us now derive from Theorem \ref{thm_spectral_cutoff} an asymptotic minimax upper bound over the Sobolev-type ellipsoids $H^\beta(R,D)$. For $m_0=\floor{\sqrt D}+1$ the bound \eqref{EqMainBound} gives
\[ \E\big[\norm{\widehat\mu^{(\tau)}-\mu}^2\big]\le C\E\big[\norm{\widehat\mu^{(t_{\mathfrak s})}-\mu}^2\big]
\]
because of $t_{\mathfrak s}\ge m_0\ge \sqrt D$. Now, $t_{\mathfrak s}(\mu)\lesssim \tmmpar$ holds with the optimal truncation index $\tmmpar$ from \eqref{Eqtmm} for $\mu\in H^\beta(R,D)$ and
if $\tmmpar\in [m_0,D]$; under these conditions $\E[\norm{\widehat\mu^{(t_{\mathfrak s})}-\mu}^2]\lesssim {\cal R}^*_{\beta,p,R}(\delta)$  is thus true
and
we obtain the following adaptive upper bound:

\begin{corollary}\label{CorSobUB}
Assume  \eqref{eq:polydecay}, $\abs{\kappa-D\delta^2}\le C_\kappa\sqrt D\delta^2$  and choose $m_0=\floor{\sqrt D}+1$. Then there is a constant $C>0$, depending only on $p$, $C_A$ and $C_\kappa$, such that for all $(\beta,R)$ with $\tmmpar\in [\sqrt D,D]$
\[\sup_{\mu\in H^\beta(R,D)}\E_\mu\big[\norm{\widehat\mu^{(\tau)}-\mu}^2\big]\le C {\cal R}^*_{\beta,p,R}(\delta).
\]
\end{corollary}

 In summary, together with the matching lower bound of Corollary \ref{CorSobLB} this shows that the stopping rule $\tau$ is sequentially minimax  adaptive.

\section{An adaptive two-step procedure}

\label{se: two-step}

\subsection{Construction and results}

The lower bounds show that, in general, there is no hope for an early stopping rule attaining the order of the (unconstrained) oracle risk if the strongly balanced oracle $t_{\mathfrak s}$ is of smaller order than $\sqrt D$. We can therefore always start the stopping rule $\tau$ at some $m_0\gtrsim \sqrt D$. If, however, immediate stopping $\tau=m_0$ occurs, we might have stopped too late in the sense that $t_{\mathfrak s}\ll m_0$. To avoid this overfitting, we propose to run a second model selection step on $\{\widehat\mu^{(0)},\ldots,\widehat\mu^{(m_0)}\}$ in the event $\tau=m_0$.

Below, we shall formalise this procedure and prove that this combined model selection indeed achieves adaptivity,
that is, its risk is controlled by an oracle inequality. While violating the initial stopping rule prescription, we still gain substantially in terms of numerical complexity.
At the heart of this twofold model selection procedure is a simple observation of independence.

\begin{lemma}\label{LemIndeptau}
The stopping rule $\tau$ is independent of the estimators $\widehat\mu^{(0)},\ldots\widehat\mu^{(m_0)}$.
\end{lemma}

\begin{proof}
By construction, $\tau$ is measurable with respect to the $\sigma$-algebra $\sigma(R_{m_0}^2,\ldots,R_D^2)=\sigma(Y_{m_0+1}^2,\ldots,Y_D^2)$ and $\widehat\mu^{(m)}$ is $\sigma(Y_1,\ldots,Y_m)$-measurable. By the independence of $(Y_1,\ldots,Y_{m_0})$ and $(Y_{m_0+1},\ldots,Y_D)$ the claim follows.
\end{proof}

For the second step, we suppose that $\widehat m\in\{0,\ldots,m_0\}$ is obtained from any model selection procedure among $\{\widehat\mu^{(0)},\ldots,\widehat\mu^{(m_0)}\}$ that satisfies with a constant $C_2\ge 1$, for any signal $\mu$, the oracle inequality
\begin{equation}\label{EqMS2Oracle}
\E[\norm{\widehat\mu^{(\widehat m)}-\mu}^2]\le C_2\Big(\min_{m\in \{0,\ldots,m_0\}}\E[\norm{\widehat\mu^{(m)}-\mu}^2]+\delta^2\Big).
\end{equation}
Such an oracle inequality holds for standard procedures, for instance the AIC-criterion
\[ \widehat m\in\argmin_{m\in \{0,\ldots,m_0\}}\Big(-\sum_{i=1}^{m}\lambda_i^{-2}Y_i^2+2\delta^2\sum_{i=1}^m \lambda_i^{-2}\Big).
\]
We refer to Section 2.3 in Cavalier and Golubev \cite{CG} for the corresponding result and further discussion. If we are interested in a weak norm oracle inequality, the AIC-criterion takes the weak empirical risk and reduces to the minimisation of $-\sum_{i=1}^{m}Y_i^2+2m\delta^2$, which is classical.
Based on the lemma and the tools developed in the previous section, we prove in Appendix \ref{SecProp42} the following oracle inequality in an asymptotic setting.

\begin{proposition}\label{Prop2ndMS}
  Assume $D\ge 3$, \eqref{eq:polydecay}, $\abs{\kappa-D\delta^2}\le C_\kappa \sqrt D \delta^2$
  and set  $m_0=\floor{128\log(D)\sqrt D}+1$. Suppose the model selector $\widehat m$ satisfies \eqref{EqMS2Oracle} with $C_2\ge 1$.
Then there is a constant $C>0$, depending only on $p$, $C_A$, $C_\kappa$ and $C_2$, such that uniformly over all signals $\mu$ the estimator
\[
\widehat\mu^{(\rho)}=\begin{cases} \widehat\mu^{(\widehat m)}, &\text{ if }\tau=m_0,\\ \widehat\mu^{(\tau)},&\text{ if }\tau>m_0\end{cases}
\]
satisfies
\[\E\big[\norm{\widehat\mu^{(\rho)}-\mu}^2\big]\le C\Big( \min_{m\in\{0,\ldots,D\}}\E\big[\norm{\widehat\mu^{(m)}-\mu}^2\big]+\delta^2\Big).
\]
\end{proposition}

In particular, $\widehat\mu^{(\rho)}$ is minimax adaptive over all Sobolev-type balls $H^\beta(R,D)$ at a usually much reduced computational complexity compared to standard model selection procedures requiring all $\widehat\mu^{(m)}$, $m=0,\ldots,D$.

\subsection{Numerical illustration}

\begin{figure}[t]
\parbox{0.48\textwidth}{\vspace*{-2mm}\includegraphics[width=0.48\textwidth]{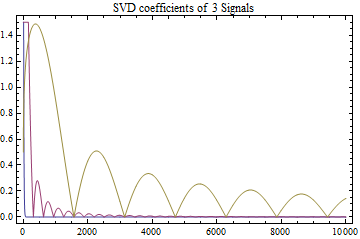}} \hfill \parbox{0.48\textwidth}{\includegraphics[width=0.48\textwidth]{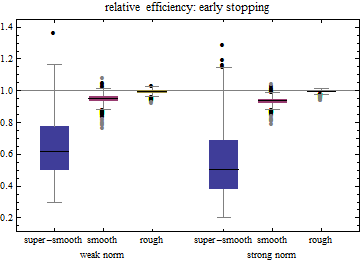}}
\caption{Left: SVD representation of a super-smooth (blue), a smooth (red) and a rough (olive) signal.
 Right: Relative efficiency for early stopping with $m_0=0$.}\label{Fig2}
\end{figure}

Let us  exemplify the procedure by some Monte Carlo results.
As a test bed we take the moderately ill-posed case $\lambda_i=i^{-1/2}$ with noise level $\delta=0.01$ and dimension $D=10\,000$. We consider early stopping at $\tau$ with $\kappa=D\delta^2=1$.

In Figure \ref{Fig2} (left), we see the SVD representation of three signals: a very smooth signal $\mu(1)$, a relatively smooth signal $\mu(2)$ and a  rough signal $\mu(3)$, the attributes coming from the interpretation via the decay of Fourier coefficients. The corresponding weakly balanced oracle indices $t_{\mathfrak w}$ are $(34,316,1356)$. The classical oracle indices in strong norm are $(43,504,1331)$. Figure \ref{Fig2} (right) shows box-plots of the relative efficiency of early stopping in 1000 Monte Carlo replications defined as
$\min_m\E[\norm{\widehat\mu^{(m)}-\mu}^2]^{1/2}/\norm{\widehat\mu^{(\tau)}-\mu}$, both for strong and weak norm. Ideally, the relative efficiency should concentrate around one. This is well achieved for the smooth and rough signals { and even better than for the corresponding Landweber results in \cite{BHR2018}.} The super-smooth case with its very small oracle risk suffers from the variability within the residual and attains on average  an efficiency of about $0.5$, meaning that its root mean squared error is about twice as large as the oracle error.
{ Let us mention that in unreported situations with higher ill-posedness  the {\it relative} efficiency is similarly good or even better.}

We are lead to consider the two-step procedure. According to Proposition \ref{Prop2ndMS} we have to choose an initial index somewhat larger than $\sqrt D$. The factor in the choice there is very conservative due to non-tight concentration bounds. For the implementation we choose $m_0$ such that for a zero signal $\mu=0$ the probability of $\{\tau>m_0\}=\{R_{m_0}^2>\kappa\}$ is about 0.01, when applying a normal approximation, that is $m_0=\floor{q_{0.99}\sqrt{2D}}+1=329$ with the $99\%$-percentile $q_{0.99}$ of $N(0,1)$. In Figure \ref{Fig3}(left) we see that with this choice for the super-smooth signal, 6 out of 1\,000 MC realisations lead to $\tau>m_0$, for the others we apply the second model selection step. The truncation for the smooth signal varies around $m_0$, and the second step is applied to about $50\%$ of the realisations. In the rough case, $\tau>m_0$ was always satisfied and no second model selection step was applied.

\begin{figure}[t]
\includegraphics[width=0.48\textwidth]{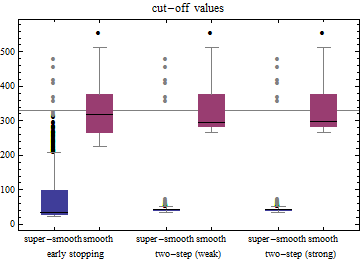} \hfill \includegraphics[width=0.48\textwidth]{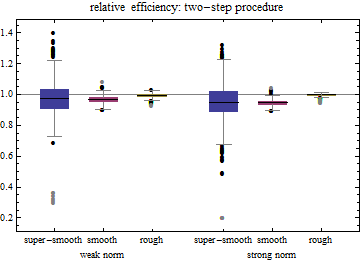}
\caption{Left: truncation levels for early stopping with $m_0=0$ and the two-step procedures with $m_0=\floor{q_{0.99}\sqrt{2D}}+1=329$, AIC in weak and strong norm. Right: Relative efficiencies for the two-step procedure.}\label{Fig3}
\end{figure}

As model selection procedure we apply the AIC-criterion, based on the weak and strong empirical norm for the weak and strong norm criterion, respectively. The results are shown in Figure \ref{Fig3}(right). We see that the efficiency for the super-smooth signal improves significantly (with the 6 outliers not being affected). The variability is still considerably higher than for the other two signals. { This phenomenon is well known for unbiased risk estimation. Especially for more strongly ill-posed problems, one should  penalise stronger, see the comparison with the risk hull approach in Cavalier and Golubev \cite{CG} and the numerical findings in Lucka {\it et al.} \cite{Bissetal}.} Here let us  rather emphasize that a pure AIC-minimisation for the super-smooth signal gives exactly the same result, apart from the 6 outliers, but requires to calculate the AIC-criterion for $D=10\,000$ indices in 1\,000 MC iterations. The two-step procedure, even for known SVD, is about
30 times faster.\\


\section{Appendix}  \label{Appendix}

\subsection{Proof of Corollary~\ref{cor:case1lower}} \label{SecPrCor22}

\begin{proof}

  For $i_0=\floor{(2C_\mu^2C_A)^{1/(2p+1)}(R^{-1}\delta)^{-2/(2\beta+2p+1)}}$, we can choose $c_2$ (in dependence of $C_\mu,C_A$)
  big enough so that our assumptions imply $i_0 \le D$ and
\[1-\frac{{\cal R}(\mu,\tau)^2}{V_{i_0+1}}\ge 1-C_\mu^2C_A \Bigg(\frac{(R^{-1}\delta)^{-2/(2\beta+2p+1)}}{i_0+1}\Bigg)^{1+2p}\ge \frac12.\]
Put $\bar\mu_i=\mu_i$ for $i\not= i_0+1$ and $\bar\mu_{i_0+1}=\tfrac12 \bar R(i_0+1)^{-\alpha}$. Then $\bar\mu\in H^\alpha(\bar R, D)$ follows from $\mu\in H^\beta(R, D)\subset H^\alpha(R, D)$ and $\bar R\ge 2R$. The bias bound $B_{i_0}^2(\bar\mu)\ge\tfrac14 \bar R^2(i_0+1)^{-2\alpha}$ inserted in Proposition \ref{prop:case1lower} yields the result.
\end{proof}

\subsection{Proof of Lemma \ref{lem:lowerstop}} \label{proof:lem:lowerstop}

\begin{proof}
With $S_m =\delta^2 \sum_{i=1}^m \lambda_i^{-2} \eps_i^2$ we obtain
\[
\cR(\mu,\tau) ^2 \ge  \mathbb E\Big[\delta^2 \sum_{i=1}^\tau \lambda_i^{-2} \eps_i^2\Big]
 \ge  \mathbb E\Big[\ind{\tau \ge m}S_m\Big].
\]
By Lemma 1 in Laurent and Massart \cite{LM}, for nonnegative weights $a_i$,  we have
\begin{equation}
    \label{lowerdevchi2}
    P\Big(\sum_{i = 1}^Da_i(\varepsilon_i^2-1) < -2\|a\|\sqrt{x} \Big)<e^{-x}.
\end{equation}
Picking $a=\delta^2 (\lambda_1^{-2},\ldots,\lambda_m^{-2})$ and $x=\log(5/4)$
so that $2\sqrt{x}\le 0.95$, with probability larger than
$1-e^{-x}=0.2$, it follows that
\[
S_m \ge \e{S_m} - 2 \norm{a} \sqrt{x} \ge \frac{\delta^2}{20} \sum_{i=1}^m \lambda_i^{-2} = \frac{V_m}{20},
\]
where we used $\norm{a} \leq  \sum_{i = 1}^ma_i= \e{S_m}=V_m$\, (observe that we could tighten
the latter inequality significantly under some additional assumptions on the singular value decay).
We now have
\begin{align*}
\cR(\mu,\tau)^2 & \ge \e{\ind{\tau \ge m} S_m} \\
&\ge V_m P(\set{\tau \ge m} \cap \set{S_m \ge V_m/20})/20\\
&\ge V_m \big(1 - P(\tau < m) - P(S_m < V_m/20)\big)/20 \\
&\ge V_m \big(0.2 - P(\tau < m)\big)/20.
\end{align*}
We deduce from this that $V_m \geq 200 \cR(\mu,\tau)^2$ implies $P(\tau \ge m) \le 0.9$.
\end{proof}

\subsection{A total variation bound for non-central $\chi^2$-laws}
\begin{lemma}\label{LemTVGamma}
Let $\vartheta = (\vartheta_1,\ldots, \vartheta_K) \in \R^K$ and $\PP_K^\theta$ be the non-central $\chi^2$-law of $X_\theta=\sum_{k=1}^K(\theta_k+Z_k)^2$ with $Z_k$ independent and standard Gaussian. Then, for $\theta,\bar\theta\in\R^K$ we have
\[ \norm{\PP_K^\theta-\PP_K^{\bar\theta}}_{TV}\le  e\frac{\abs{\norm{\theta}^2-\norm{\bar\theta}^2}+\sqrt{8/\pi}\abs{\norm{\theta}-\norm{\bar\theta}}}{\sqrt{\pi K}},\]
For $\norm{\theta}+\norm{\bar\theta}\ge \frac{\sqrt{8}e}{2\pi-\sqrt{\pi}e}\approx 5.248$ this bound simplifies to
\[ \norm{\PP_K^\theta-\PP_K^{\bar\theta}}_{TV}\le 2 \frac{\abs{\norm{\theta}^2-\norm{\bar\theta}^2}}{\sqrt{K}}.\]
\end{lemma}
\begin{proof}
Writing $\theta=(\theta_k),Z=(Z_k)\in\R^k$ we see by orthogonal transformation that $X_\theta=\norm{\theta}^2+2\scapro{\theta}{Z}+\norm{Z}^2$ equals in law
$X'_\theta=\norm{\theta}^2+2\norm{\theta}Z_1'+\norm{Z'}^2$ with $Z_1',\ldots,Z_K'\sim N(0,1)$ i.i.d. We can therefore first consider the conditional law $\QQ_K^\theta(z)$ of $\PP_K^\theta$ given $\{Z_1'=z\}$, which is nothing but the $\chi^2(K-1)$-distribution translated by $\norm{\theta}^2+2\norm{\theta}z+z^2$.

If $f_p$ denotes the $\chi^2(p)$-density, then we have for any $t>0$ that $f_p(x-t)>f_p(x)$ holds iff $x\ge x_t=\frac{t}{1-e^{-t/(p-2)}}$. Thus, we obtain
\begin{multline*}
  \, \int_0^\infty\abs{f_p(x-t)-f_p(x)}\,dx \\
  \begin{aligned}
=&\, 2\int_0^\infty \big(f_p(x-t)-f_p(x)\big)_+\,dx\\
=&\, \frac{2^{1-p/2}}{\Gamma(p/2)}\int_{x_t}^\infty \big((1-t/x)^{p/2-1}e^{t/2}-1\big)x^{p/2-1}e^{-x/2}dx\\
=&\, \frac{2^{1-p/2}}{\Gamma(p/2)}\int_{x_t-t}^{x_t} x^{(p-2)/2}e^{-x/2}dx\\
\le&\, \frac{2^{(2-p)/2}}{\Gamma(p/2)} t(p-2)^{(p-2)/2}e^{-(p-2)/2},
  \end{aligned}
  \end{multline*}
knowing that $x=p-2$ is the mode of $f_p$. Stirling's formula guarantees $\Gamma(x)\ge \sqrt{2\pi/ x}(x/e)^x$ for all $x>0$ such that the last expression is always bounded by $t(\pi p)^{-1/2}e$. This yields
\[ \norm{\QQ_K^\theta(z)-\QQ_K^{\bar\theta}(z)}_{TV}\le e(\pi K)^{-1/2}\big|\norm{\theta}^2-\norm{\bar\theta}^2+2(\norm{\theta}-\norm{\bar\theta})z\big|.
\]
Taking expectation with respect to $Z_1'\sim N(0,1)$ we conclude
\[ \norm{\PP_K^\theta-\PP_K^{\bar\theta}}_{TV}\le e(\pi K)^{-1/2}\big|\norm{\theta}-\norm{\bar\theta}\big|\E\big[\big|\norm{\theta}+\norm{\bar\theta}+2Z_1'\big|\big].
\]
Using the triangle inequality and $\E[\abs{Z_1'}]=\sqrt{2/\pi}$, the upper bound follows.
\end{proof}

\subsection{Proof of Corollary~\ref{cor: residual LB1}}\label{SecCor25}

\begin{proof}
Set $\bar\mu_i=\mu_i$ for $i\not= i_0+1$ and $\bar\mu_{i_0+1}^2=\mu_{i_0+1}^2+\tfrac14 \bar R^2 (i_0+1)^{-2\alpha}$ for some $i_0\in\{1,\ldots,D\}$, so that  $\bar\mu\in H^\alpha(\bar R, D)$ and condition $(a)$ of Proposition \ref{prop:case2lower} is satisfied. If
\[\text{(b'): }\tfrac{C_A^2}4\bar R^2(i_0+1)^{-2(\alpha+p)} \le 0.025\delta^2\sqrt{D-i_0}\]
holds, then condition (b)  of Proposition \ref{prop:case2lower} is ensured, whereas
\[\text{(c'): }\bar R^2(i_0+1)^{-2(\alpha+p)}\ge 2 C_A^2 5.25^2\delta^2\]
implies condition (c)  of Proposition \ref{prop:case2lower}. Finally, for
\[\text{(d'): }i_0\ge\floor{(200(1+2p)C_A^{2} C_\mu^2)^{1/(2p+1)}(R^2\delta^{-2})^{1/(2\beta+2p+1)}},\]
we have $V_{i_0+1}\ge 200{\cal R}(\mu,\tau)^2$. Hence, by Proposition \ref{prop:case2lower},
(b')-(c')-(d') imply
\[{\cal R}(\bar\mu,\tau)^2\ge 0.05 B_{i_0}^2(\bar\mu)\ge \tfrac{0.05}{4}\bar R^2(i_0+1)^{-2\alpha}.\]
For
$i_0=\big\lfloor C_0\max\big((\bar R^2 \delta^{-2}/\sqrt{D})^{1/(2\alpha+2p)},(R^2\delta^{-2})^{1/(2\beta+2p+1)}\big)\big\rfloor$ with some suitably large constant $C_0>0$, depending only on $C_\mu,C_A,p$, and for $D\ge 2i_0$,  conditions (b') and (d') are satisfied. To check condition (c'), a sufficient condition is
$i_0+1 \le ( \bar R^2 \delta^{-2}/(56 C_A^2) )^{1/(2\alpha +2p)}$.
The first term in the maximum defining $i_0$ satisfies this condition
(here again using $D\ge 2 i_0$) provided $R^{-1}\delta$ is smaller than a suitable constant
$c_2'$ depending on $C_A, C_\mu, \alpha,p$. The second term in the maximum defining
$i_0$ satisfies the sufficient condition
\[
C_0 (R^2 \delta^{-2})^{1/(2\beta + 2p +1)}
\le C_0 (\bar R^2 \delta^{-2})^{1/(2\alpha + 2p +1)}
\le (\bar R^2 \delta^{-2}/ (56 C_A^2)  )^{1/(2\alpha +2p)},
\]
again as soon as $R^{-1}\delta$ is smaller than a suitable constant $c_2''$ depending
on the same parameters as $c_2'$. Finally, putting $c_2=\min(c_2',c_2'',1)$ and unwrapping the condition $D\ge 2i_0$,
yields (using $R\delta^{-1}\ge 1$) the sufficient condition $D\ge c_3' (\bar R^{2} \delta^{-2})^{1/(2\alpha+2p+1/2)}$,
which is equivalent to the assumption
$D\ge c_3 \tmm_{\alpha-\frac{1}{4},p,\bar R}(\delta)$ postulated in the statement,
for suitable $c_3',c_3$ depending on $C_\mu,C_A,\alpha,p$.
This yields the result.
\end{proof}

\subsection{Proof of Proposition \ref{PropBiasbound}}\label{SecProp35}

\begin{proof}
By partial summation, we deduce from $B_{\tau,\lambda}^2(\mu)>B_{t^\ast,\lambda}^2(\mu)$, which implies $\tau\le\floor{t^\ast}$:
\begin{align*}
  \E[(B_{\tau,\lambda}^2(\mu)-B_{t^\ast,\lambda}^2(\mu))_+]
&= \sum_{m=m_0}^{\floor{t^\ast}}(B_{m,\lambda}^2(\mu)-B_{t^\ast,\lambda}^2(\mu)) P(\tau=m)\\
  &
   = \sum_{m=m_0}^{\floor{t^\ast}}
    \sum_{k=m}^{\floor{t^\ast}} (B_{k,\lambda}^2(\mu)-B_{(k+1) \wedge t^\ast,\lambda}^2(\mu)) P(\tau=m)\\
&= \sum_{m=m_0}^{\floor{t^\ast}}(B_{m,\lambda}^2(\mu)-B_{(m+1)\wedge t^\ast,\lambda}^2(\mu))P(\tau\le m).
\end{align*}
In the case $t^\ast=m_0$ all expressions evaluate to zero because of $\tau\ge t^\ast$ and we  suppose $t^\ast>m_0$ from now on, so that \eqref{EqRtstarkappa} holds.
For $m_0\le m<t^\ast$ we have  $\{\tau\le m\}=\{R_m^2\le \kappa\}$, $\E[R_m^2]\geq \kappa$ and  by Lemma 1 in \cite{LM} (recall \eqref{lowerdevchi2} in the proof of Lemma \ref{lem:lowerstop}) together with $P(Z<-x)\le e^{-x^2/(2\sigma^2)}$ for $Z\sim N(0,\sigma^2)$, $x\ge 0$ we obtain  the bound
\begin{align*}
&P(R_m^2\le\kappa) =P\Big(\sum_{i=m+1}^D\big(\delta^2(\eps_i^2-1)+2\lambda_i\mu_i\delta\eps_i\big)\le -(\E[R_m^2]-\kappa)\Big)\\
&\le P\Big(\sum_{i=m+1}^D\delta^2(\eps_i^2-1)\le -\frac{\E[R_m^2]-\kappa}{2}\Big) +P\Big(\sum_{i=m+1}^D\lambda_i\mu_i\delta\eps_i\le -\frac{\E[R_m^2]-\kappa}{4}\Big)\\
&\le \exp\Big(-\frac{(\E[R_m^2]-\kappa)^2}{16\delta^4(D-m)}\Big)+ \exp\Big(-\frac{(\E[R_m^2]-\kappa)^2}{32\delta^2B_{m,\lambda}^2(\mu)}\Big)\\
&\le F(B_{m,\lambda}^2(\mu)-B_{t^\ast,\lambda}^2(\mu)),
\end{align*}
where we use $\E[R_m^2]-\kappa\ge B_{m,\lambda}^2(\mu)-B_{t^\ast,\lambda}^2(\mu)$
for $m<t^*$ (see \eqref{expresidual}, \eqref{EqRtstarkappa}), and put
\[F(z):= \exp\Big(-\frac{z^2}{16\delta^4D}\Big)+ \exp\Big(-\frac{z^2}{32\delta^2(B_{t^\ast,\lambda}^2(\mu)+z)}\Big),\quad z\ge 0.
\]
We conclude by monotonicity of $B_{\cdot,\lambda}^2(\mu)$ and $F$  via a Riemann-Stieltjes sum approximation:
\begin{align*}
&\E\big[\big(B_{\tau,\lambda}^2(\mu)-B_{t^\ast,\lambda}^2(\mu)\big)_+\big]\\
&\le \sum_{m=m_0}^{\floor{t^\ast}}(B_{m,\lambda}^2(\mu)-B_{(m+1)\wedge t^\ast,\lambda}^2(\mu)) F(B_{m,\lambda}^2(\mu)-B_{t^\ast,\lambda}^2(\mu)) \\
&\le \int_{B_{t^\ast,\lambda}^2(\mu)}^{B_{m_0,\lambda}^2(\mu)} F(y-B_{t^\ast,\lambda}^2(\mu))\,dy\\
&\le \int_0^\infty F(z)\,dz\\
&\le \sqrt{4\pi\delta^4D}+\int_0^{B_{t^\ast,\lambda}^2(\mu)} \exp\Big(-\frac{z^2}{64\delta^2B_{t^\ast,\lambda}^2(\mu)}\Big)\,dz+\int_{B_{t^\ast,\lambda}^2(\mu)}^\infty  \exp\Big(-\frac{z}{64\delta^2}\Big)\,dz\\
&\le \sqrt{4\pi\delta^4D}+\sqrt{16\pi\delta^2B_{t^\ast,\lambda}^2(\mu)}+64\delta^2 \\
&\le \big(17\sqrt{D}+64\big)\delta^2+B_{t^\ast,\lambda}^2(\mu) D^{-1/2},
\end{align*}
using  $4\sqrt\pi\delta B_{t^\ast,\lambda}(\mu)\le 4\pi \sqrt{D}\delta^2+D^{-1/2}B_{t^\ast,\lambda}^2(\mu)$ by the binomial identity and $\sqrt{4\pi}+4\pi\le 17$ in the last line.
\end{proof}

\subsection{Proof of Proposition \ref{PropVartautstar}}\label{SecProp37}

\begin{proof}
By the Cauchy-Schwarz inequality and $\E[\eps_m^4]^{1/2}=\sqrt 3$, we have
\begin{align*}
  \E\Big[(S_\tau-S_{\floor{t^\ast}+1})_+\Big]&= \delta^2
  \sum_{m=\floor{t^\ast}+2}^{D}\lambda_m^{-2}\E[\eps_m^2{\bf 1}(\tau\ge m)]\\
&\le \sqrt 3 \delta^2 \sum_{m=\floor{t^\ast}+2}^{D}\lambda_m^{-2}P(\tau\ge m)^{1/2}.
\end{align*}
For $m\ge t^\ast+1\ge m_0+1$ we have  $\{\tau\ge m\}=\{R_{m-1}^2> \kappa\}$,
$\E[R^2_{m-1}]  \leq \kappa$ and by $P(\sum_{i=1}^Da_i(\varepsilon_i^2-1) > x) \leq \exp\big(-x^2/(4\|a\|^2+4x \max_i a_i )\big)$ for $a_i\ge 0$ (Lemma 1 in \cite{LM}) together with $P(Z>x)\le e^{-x^2/(2\sigma^2)}$ for $Z\sim N(0,\sigma^2)$, $x\ge 0$, we obtain the bound
\begin{align*}
&P(R_{m-1}^2>\kappa) =P\Big(\sum_{i=m}^D\big(\delta^2(\eps_i^2-1)+2\lambda_i\mu_i\delta\eps_i\big)>\kappa -\E[R_{m-1}^2]\Big)\\
&\le P\Big(\sum_{i=m}^D\delta^2(\eps_i^2-1)\geq \frac{\kappa-\E[R_{m-1}^2]}{2}\Big) +P\Big(\sum_{i=m}^D\lambda_i\mu_i\delta\eps_i\geq \frac{\kappa-\E[R_{m-1}^2]}{4}\Big)\\
&\le \exp\Big(-\frac{(\kappa-\E[R_{m-1}^2])^2}{16\delta^4(D-m+1)+ 8\delta^2(\kappa-\E[R_{m-1}^2])}\Big)+ \exp\Big(-\frac{(\kappa-\E[R_{m-1}^2])^2}{32\delta^2B_{m-1,\lambda}^2(\mu)}\Big).
\end{align*}
For the numerator, we use the lower bound
$$\kappa-\E[R_{m-1}^2]\ge \kappa-\E[R_{t^\ast}^2]+\delta^2(m-1-t^\ast)\ge \delta^2(m-1-t^\ast).$$
For the denominators, we use $16\delta^4(D-m+1)+ 8\delta^2(\kappa-\E[R_{m-1}^2]) \leq 16\delta^4D+ 8\delta^2\kappa$ for the first term and $32\delta^2B_{m-1,\lambda}^2(\mu)\le 32\delta^2B_{t^\ast,\lambda}^2(\mu)\le 32\delta^2\kappa$ for the second term. We arrive at
\begin{align*}
\E\Big[(S_\tau-S_{\floor{t^\ast}+1})_+\Big]
&\le 2\sqrt 3 \delta^2 \sum_{m=\floor{t^\ast}+2}^D \lambda_{m}^{-2}\exp\Big(-\frac{(m-1-t^\ast)^2}{16D+32\kappa
 \delta^{-2}}\Big).
\end{align*}
We add $\E[(S_{\floor{t^\ast}+1}-S_{t^\ast})_+]\le 2\sqrt 3 \delta^2\lambda_{\floor{t^\ast}+1}^{-2}$ and note
$(S_\tau-S_{t^\ast})_+\le S_D$, which gives the trivial bound $\E[S_D]=D\delta^2$. Under the polynomial eigenvalue decay this yields the bound
\[ r_{V,\tau}\le 2\sqrt 3 C_A^{2}\Big(1+\sum_{k\ge 0} (t^\ast+1+k)^{2p}e^{-k^2/(16D+32\kappa\delta^{-2})}\Big)\wedge D.\]
In the sequel  '$\lesssim$', '$\gtrsim$' denote inequalities up to a factor only depending on $p$. A Riemann sum approximation shows for any $R>0$
\begin{align*}
\sum_{k\ge 0} e^{-k^2/R} &\le 1+\int_0^\infty e^{-x^2/R}dx \lesssim \sqrt{R}.
\end{align*}
Similarly, we obtain
\begin{align*}
\sum_{k\ge 0} k^{2p}e^{-k^2/R} &\le \int_0^\infty (1+x)^{2p}e^{-x^2/R}dx
\lesssim R^{p+1/2}.
\end{align*}
This yields
\[ r_{V_\tau}\lesssim C_A^{2}\Big((t^\ast)^{2p}\sqrt{D+\kappa\delta^{-2}}+(D+\kappa\delta^{-2})^{p+1/2}\Big).\]
On the other hand, we have
\[V_{t^\ast}=\delta^2\Big(\sum_{m=1}^{\floor{t^\ast}}\lambda_m^{-2}+(t^\ast-\floor{t^\ast})\lambda_{\floor{t^\ast}+1}^{-2}\Big) \gtrsim \delta^2 C_A^{-2}(t^\ast)^{2p+1},\]
implying the result with a suitable constant $C_p$.
\end{proof}

\subsection{Proof of Proposition \ref{Prop2ndMS}} \label{SecProp42}

\begin{proof}
  In this proof '$\lesssim$' denotes an inequality holding up to factors depending only on $p$, $C_A$, $C_\kappa$ and $C_2$;
  similarly, '$\sim$' denotes a two-sided inequality holding up to factors depending on these parameters.
In the case $t_{\mathfrak s}> m_0$ we  use the independence of $\tau$ from $\widehat\mu^{(0)},\ldots\widehat\mu^{(m_0)}$ by Lemma \ref{LemIndeptau}
to obtain
\begin{align*}
\E[\norm{\widehat\mu^{(\rho)}-\mu}^2{\bf 1}(\tau=m_0)]&= \E[\norm{\widehat\mu^{(\widehat m)}-\mu}^2]P(\tau=m_0)\\
& \le C_2\Big(\E[\norm{\widehat\mu^{(m_0)}-\mu}^2]+\delta^2\Big)P(\tau=m_0)\\
& = C_2\Big(\E[\norm{\widehat\mu^{(\tau)}-\mu}^2{\bf 1}(\tau=m_0)]+\delta^2 P(\tau=m_0)\Big).
\end{align*}
On $\{\tau>m_0\}$ we have $\rho=\tau$ and we apply Theorem \ref{thm_spectral_cutoff} with $t_{\mathfrak s}\ge m_0>\sqrt{D}$ to get
\[ \E[\norm{\widehat\mu^{(\rho)}-\mu}^2{\bf 1}(\tau>m_0)]\le \E[\norm{\widehat\mu^{(\tau)}-\mu}^2]\lesssim \E[\norm{\widehat\mu^{(t_{\mathfrak s})}-\mu}^2].
\]
Because of $t_{\mathfrak s}>m_0$ we have $\E[\norm{\widehat\mu^{(t_{\mathfrak s})}-\mu}^2]\le 2\min_{t\in[0,D]}\E[\norm{\widehat\mu^{(t)}-\mu}^2]$.
This gives the result in this case.

Next, consider the case $t_{\mathfrak s}= m_0$ where  $B_{m_0}^2(\mu)\le V_{m_0}$. Then the estimator $\widehat\mu^{(m_{\mathfrak s})}$ with $m_{\mathfrak s}\in\{0,1,\ldots,m_0\}$ from \eqref{Eqms} satisfies
\[\E\big[\norm{\widehat\mu^{(m_{\mathfrak s})}-\mu}^2\big]\le 2\max_i\frac{\lambda_i^2}{\lambda_{i+1}^2} \min_{m\in\{0,\ldots,D\}}\E\big[\norm{\widehat\mu^{(m)}-\mu}^2\big],
\]
noting that the factor $\max_i\lambda_i^2/\lambda_{i+1}^2$ comes from the discretisation $m_{\mathfrak s}$ of the balanced oracle and is bounded by
$C_A^42^{2p}\lesssim 1$. Given the independence of $\tau$ from $\{\widehat\mu^{(0)},\ldots,\widehat\mu^{(m_0)}\}$ by Lemma \ref{LemIndeptau} and the properties of the model selector $\widehat m$, we have
\begin{align*}
E[\norm{\widehat\mu^{(\rho)}-\mu}^2{\bf 1}(\tau=m_0)] &\le C_2\Big(\E[\norm{\widehat\mu^{(m_{\mathfrak s})}-\mu}^2]+\delta^2\Big)\\
&\lesssim \min_{m\in\{0,\ldots,D\}}\E[\norm{\widehat\mu^{(m)}-\mu}^2]+\delta^2.
\end{align*}
For $m_{\mathfrak s}\in [m_0/2,m_0]$
\[ \E[\norm{\widehat\mu^{(m_0)}-\mu}^2]\le B_{m_{\mathfrak s}}^2(\mu)+V_{m_0}\lesssim \E[\norm{\widehat\mu^{(m_{\mathfrak s})}-\mu}^2]\]
follows from $V_{m_0}\thicksim \delta^2 m_0^{2p+1}\thicksim V_{m_{\mathfrak s}}$. By  Theorem \ref{thm_spectral_cutoff} with $t_{\mathfrak s}=m_0$ this gives
\begin{align*}
\E[\norm{\widehat\mu^{(\rho)}-\mu}^2{\bf 1}(\tau>m_0)] &\lesssim \E[\norm{\widehat\mu^{(m_0)}-\mu}^2]\lesssim \E[\norm{\widehat\mu^{(m_{\mathfrak s})}-\mu}^2]\\
& \lesssim \min_{m\in\{0,\ldots,D\}}\E[\norm{\widehat\mu^{(m)}-\mu}^2].
\end{align*}
For $m_{\mathfrak s}<m_0/2$ we obtain by $\{\tau>m_0\}=\{R_{m_0}^2>\kappa\}$, $S_\tau\le S_D$, $B_\tau^2(\mu)\le B_{m_0}^2(\mu)$ and the Cauchy-Schwarz inequality:
\begin{align*}
\E[\norm{\widehat\mu^{(\rho)}-\mu}^2{\bf 1}(\tau>m_0)] &\le\E[(B_{m_0}^2(\mu)+S_D){\bf 1}(R_{m_0}^2>\kappa)]\\
&\le B_{m_0}^2(\mu)P(R_{m_0}^2>\kappa)+\E[S_D^2]^{1/2}P(R_{m_0}^2>\kappa)^{1/2}.
\end{align*}
We have $B_{m_0}^2(\mu)\le B_{m_{\mathfrak s}}^2(\mu)\le V_D$ and $\E[S_D^2]^{1/2}\lesssim \E[S_D] = V_D$ (by comparison of Gaussian moments), so that
\[ \E[\norm{\widehat\mu^{(\rho)}-\mu}^2{\bf 1}(\tau>m_0)]\lesssim V_DP(R_{m_0}^2>\kappa)^{1/2}\thicksim \delta^2DP(R_{m_0}^2>\kappa)^{1/2}.
\]
Observing $m_{\mathfrak w}:=\min\{m\ge 0\,|\, B_{m,\lambda}^2(\mu)\le V_{m,\lambda}\}\le m_{\mathfrak s}<m_0/2$ we obtain
\[\E[R_{m_0}^2]-\kappa=B_{m_0,\lambda}^2(\mu)-V_{m_0,\lambda}\le B_{m_{\mathfrak w},\lambda}^2(\mu)-m_0\delta^2\le -(m_0/2)\delta^2.\]
As in the proof of Proposition \ref{PropVartautstar} we therefore find
\[ P(R_{m_0}^2>\kappa)\le \exp\Big(-\frac{m_0^2}{64(D-m_0)+ 16m_0}\Big)+ \exp\Big(-\frac{m_0^2\delta^2}{128B_{m_0,\lambda}^2(\mu)}\Big).
\]
By the choice of $m_0$ and $B_{m_0,\lambda}^2(\mu)\le B_{m_0/2,\lambda}^2(\mu)\le (m_0/2)\delta^2$,
 using  $D\geq 3\Rightarrow \log D \geq 1 $, we deduce
\[ P(R_{m_0}^2>\kappa)\le 2\exp\big(-2\log D\big)=2D^{-2}.
\]
Insertion of this bound yields
$ \E[\norm{\widehat\mu^{(\rho)}-\mu}^2{\bf 1}(\tau>m_0)]\lesssim \delta^2$,
which accomplishes the proof for the case $t_{\mathfrak s}= m_0$ and $m_{\mathfrak s}\le m_0/2$.
\end{proof}

\section*{Acknowledgement}
Very instructive hints and questions by Thorsten Hohage, Alexander Golden\-shluger, Peter Math\'e and two anonymous referees are gratefully acknowledged.
 This research was supported by the DFG via SFB 1294 {\it Data Assimilation} and via Research Unit 1735 {\it Structural Inference in Statistics}, which has made possible a six months visit by MH to Humboldt-Universit\"at zu Berlin.

\bibliographystyle{plain}       
\bibliography{biblioBHR}           

\checknbdrafts

\end{document}